\documentclass[11,a4paper]{article}
\usepackage[english]{babel}
\usepackage[utf8]{inputenc}
\usepackage{amsmath,amssymb,graphicx,enumerate,latexsym,theorem,mathtools,epstopdf,comment,placeins}
\usepackage{color}
\usepackage{etoolbox}
\usepackage{multirow}
\usepackage[
  hidelinks,
  pagebackref,
  bookmarksopen,
  bookmarksnumbered,
]{hyperref}
\hypersetup{
  pdfauthor={Miguel de Benito Delgado},
  pdfauthor={Bernd Schmidt}
  % ...
}
\catcode`\<=\active \def<{
\fontencoding{T1}\selectfont\symbol{60}\fontencoding{\encodingdefault}}

\newcommand{\TeXmacs}{T\kern-.1667em\lower.5ex\hbox{E}\kern-.125emX\kern-.1em\lower.5ex\hbox{\textsc{m\kern-.05ema\kern-.125emc\kern-.05ems}}}
\newcommand{\assign}{\coloneqq}

\newcommand{\chapter}[1]{\medskip\bigskip

\noindent\textbf{\huge #1}}

\newcommand{\mathd}{\,\mathrm{d}}

\newcommand{\nin}{\not\in}
\newcommand{\nobracket}{}
\newcommand{\nocomma}{}
\newcommand{\nospace}{}

\newcommand{\tmem}[1]{{\em #1\/}}
\newcommand{\tmname}[1]{\textsc{#1}}
\newcommand{\tmop}[1]{\ensuremath{\operatorname{#1}}}

\newcommand{\tmtextit}[1]{{\itshape{#1}}}
\newenvironment{enumeratealpha}{\begin{enumerate}[a{\textup{)}}] }{\end{enumerate}}

\newenvironment{proof}{\noindent\textbf{Proof\ }}{\hspace*{\fill}$\Box$\medskip}

\newenvironment{tmparmod}[3]{\begin{list}{}{\setlength{\topsep}{0pt}\setlength{\leftmargin}{#1}\setlength{\rightmargin}{#2}\setlength{\parindent}{#3}\setlength{\listparindent}{\parindent}\setlength{\itemindent}{\parindent}\setlength{\parsep}{\parskip}} \item[]}{\end{list}}

\newenvironment{proof*}[1]{\noindent\textbf{#1\ }}{\hspace*{\fill}$\Box$\medskip}

\newtheorem{definition}{Definition}

\newtheorem{proposition}{Proposition}
\newtheorem{lemma}{Lemma}
\newtheorem{theorem}{Theorem}

\newcounter{nnnotation}

\newtheorem{notation*}[nnnotation]{Notation}
\newcounter{nnnote}

{\theorembodyfont{\rmfamily}\newtheorem{note*}[nnnote]{Note}}
{\theorembodyfont{\rmfamily}\newtheorem{remark}{Remark}}
{\theorembodyfont{\rmfamily}\newtheorem{problem}{Problem}}
\newcounter{stepcount}
\newenvironment{step}[2]{\refstepcounter{stepcount}\vspace{1ex}\noindent\underline{Step \arabic{stepcount}: {#1}}{\\ \vspace{1ex}#2}}

\AtBeginEnvironment{proof}{\setcounter{stepcount}{0}}

\newcommand{\grs}{\ensuremath{\nabla_{s}}}
\newcommand{\gra}{\ensuremath{\nabla_{a}}}

\newcommand{\applicationspace}[1]{\quad}

\begin{document}

\begin{center}
\begin{LARGE}
Energy minimising configurations of pre-strained multilayers
\end{LARGE}
\\[0.5cm]
\begin{large}
Miguel de Benito Delgado\footnote{Universit{\"a}t Augsburg, Germany, {\tt m.debenito.d@gmail.com}} and 
Bernd Schmidt\footnote{Universit{\"a}t Augsburg, Germany, {\tt bernd.schmidt@math.uni-augsburg.de}}
\end{large}
\\[0.5cm]
\today
\\[1cm]
\end{center}

%\subjclass[2010]{XXX, XXX, XXX}
%\keywords{XXX, XXX, XXX} 
 
\begin{abstract}
We investigate energetically optimal configurations of thin structures with a pre-strain. Depending on the strength of the pre-strain we consider a whole hierarchy of effective plate theories with a spontaneous curvature term, ranging from linearised Kirchhoff to von K{\'a}rm{\'a}n to linearised von K{\'a}rm{\'a}n theories. While explicit formulae are available in the linearised regimes, the von K{\'a}rm{\'a}n theory turns out to be critical and a phase transition from cylindrical (as in linearised Kirchhoff) to spherical (as in von linearised K{\'a}rm{\'a}n) configurations is observed there. We analyse this behavior with the help of a whole family $(\mathcal{I}^{\theta}_{\rm vK})_{\theta \in (0,\infty)}$ of effective von K{\'a}rm{\'a}n functionals which interpolates between the two linearised regimes. We rigorously show convergence to the respective explicit minimisers in the asymptotic regimes $\theta \to 0$ and $\theta \to \infty$. Numerical experiments are performed for general $\theta \in (0,\infty)$ which indicate a stark transition at a critical value of $\theta$. 
\end{abstract}

\tableofcontents

%-----------------------------------------------------------------------------------------
%-----------------------------------------------------------------------------------------
%-----------------------------------------------------------------------------------------
\section{Introduction}\label{sec:intro}
%-----------------------------------------------------------------------------------------

The topic of this paper is motivated by experimental observations on optimal energy configurations in thin (heterogeneous) structures with a pre-strain. The simplest example of such a structure is the classical bimetallic strip which consists of two strips of different materials with different thermal expansion coefficients joined together throughout their length. If heated or cooled, due to the misfit of equilibria, internal stresses develop. The flat reference configuration is no longer optimal and the strip bends in order to reduce elastic energy. This behavior can effectively be modelled with a 1d energy functional comprising a temperature dependent spontaneous curvature term. 

In this paper we will investigate thin layers whose two lateral dimensions are much larger than their very small height and whose flat reference configuration is subject to internal stresses (one speaks of pre-strained or pre-stressed bodies). Examples of such structures are heated materials (with inhomogeneous expansion coefficients as in the bimetallic strip referred to above or homogeneous materials with a temperature gradient), crystallisations on top of a substrate as in epitaxially grown layers, or biological materials whose internal misfit is caused by swelling and growing tissue. Our main focus will be on multilayered heterogeneous plates, for which the effective plate theories have been provided in {\cite{DeBenitoSchmidt:19a}}. Our findings, however, apply equally to different situations as long as they are described by the same effective functionals, cf.\ Remark \ref{rmk:Wh-ass-and-ex} below. 
 
As a matter of fact, the situation is much more complicated and interesting for two dimensional plates than for one dimensional strips. It has been found that the assumed shape depends on the strength of the pre-strain and the aspect ratio of the specimen: Large pre-strains in very thin layers tend to cause cylindrical shapes whereas smaller pre-strains in thicker layers lead to spherical caps, \cite{SalamonMasters93,SalamonMasters95,FinotSuresh96,Freund00,KimLombardo08,EgunovKorvinkLuchnikov16}. To explain this observation one argues that locally the energy is best released if a spherical shape is assumed. If, however, the aspect ratio is very small, i.e., the lateral dimensions are very large compared to the thickness, then this leads to geometric incompatibilities: non-zero Gau{\ss} curvature introduces a change of the metric which by far has too high elastic energy. In contrast, cylindrical shapes do not lead to such incompatibilities. 

A thorough theoretical understanding of this mechanism through which `misfit' of equilibria is converted into mechanical displacement is not only interesting from a mathematical point of view. In view of applications it has proved to constitute a convenient and feasible method to access and manipulate objects even at the nanoscale. By way of example we mention experiments on the self-organised fabrication of nano-scrolls, as reported in {\cite{SchmidtEberl:01,Grundmann:03,Paetzelt-etal:06}}. 

The aim of this paper is to shed light on the geometry of energetically optimal configurations of pre-strained heterostructures with the help of two-dimensional plate theories. More precisely, we consider effective plate theories for multilayers with reference configuration $\Omega_h = \omega \times (-h/2, h/2)$, $0 < h \ll 1$, whose (small) misfit pre-strain is described by a matrix $h^{\alpha - 1} B^h$, scaling with $h$. 

The particular case $\alpha = 2$ with a misfit of the order $h$ of the aspect ratio has been investigated in {\cite{schmidt_minimal_2007,schmidt_plate_2007, BartelsBonitoNochetto:17}}. The appropriate plate theory is the nonlinear Kirchhoff theory (in the finite bending regime) and energy minimizers turned out to be (portions of) cylinders whose possible winding directions and radii are determined explicitly. Therefore, in order to be able to encounter different behavior one has to consider weaker scalings of the misfit. 

In {\cite{DeBenitoSchmidt:19a}} -- based on the homogeneous case explored in {\cite{friesecke_hierarchy_2006}} -- we have found a whole hierarchy of effective plate theories for the scalings $\alpha > 2$. Suitably rescaled, one obtains only three different limiting plate theories: the linearised Kirchhoff theory for $\alpha \in (2, 3)$, the von K{\'a}rm{\'a}n theory for $\alpha = 3$ and the linearised von K{\'a}rm{\'a}n theory for $\alpha > 3$. With a view to our present investigation, we have moreover derived a fine scale $\theta$ in the critical von K{\'a}rm{\'a}n scale which interpolates continuously between the the two linearised theories. 

For such small misfits one is lead to describe a deformation $y^h : \Omega_h \to \mathbb{R}^3$ in terms of the scaled and averaged in-plane, respectively, out-of-plane displacements
\begin{align}\label{def:scaled-in-out-of-plane-displacements} 
\begin{split}
  u_{i}^h (x_1, x_2) 
  & \assign \frac{1}{(\sqrt{\theta}h)^{\gamma}}  \int_{- 1 / 2}^{1 / 2} \big( y^h_i (x_1, x_2, x_3) - x_i \big) \mathd x_3,\quad i = 1,2, \\
  v^h (x_1, x_2) 
  & \assign \frac{1}{(\sqrt{\theta}h)^{\alpha - 2}}  \int_{- 1 / 2}^{1 / 2} y^h_3 (x_1, x_2, x_3) \mathd x_3,
\end{split}
\end{align}
where $\theta \equiv 1$ unless $\alpha = 3$ and 
\[ \gamma = \left\{\begin{array}{rll}
     2 (\alpha - 2) & \text{if} & \alpha \in (2, 3],\\
       \alpha - 1 & \text{if} & \alpha \ge 3.
   \end{array}\right. \]
A limiting plate theory in terms of the limiting quantities $(u, v)$ is then derived as the $\Gamma$-limit of the 3d nonlinearly elastic energy, rescaled by $h^{1 - 2 \alpha}$, cf.\ {\cite{DeBenitoSchmidt:19a}}. For a minimizer $(u, v)$ of the limiting theory one obtains the shape of an optimal configuration at finite $0 < h \ll 1$: After descaling, its $x_3$-averaged displacement is given approximately by 
\[ (x_1, x_2) 
   \mapsto \big( (\sqrt{\theta}h)^{\gamma} u(x_1, x_2), (\sqrt{\theta}h)^{\alpha - 2} v(x_1, x_2) \big). \] 
Since $\gamma > \alpha - 2$, the in-plane components are indeed much smaller than the out-of-plane component. In his sense, the shape is to leading order described by $v : \omega \to \mathbb{R}$ only.  

In the linearised regimes our results give the following picture: If $\alpha < 3$, degenerate parabolas (infinitesimal parts of cylinders) are seen to be optimal, whereas for $\alpha > 3$, non-degenerate parabolas (infinitesimal parts of an elliptical cap) are energy minimizers. Only in the latter case, however, the minimizer is unique (up to affine terms). Yet, even in case $\alpha < 3$ it turns out the geometric shape is uniquely determined as an infinitesimal part of a cylinder while the winding direction and radius may have several optimal values. In both cases we explicitly determine these minimizers. A basic observation shows that for $\alpha = 3$ these configurations are still asymptotically optimal in the `almost linearised' regimes $\theta \gg 1$ and $\theta \ll 1$, respectively.

The von K{\'a}rm{\'a}n regime is much more subtle.  We focus on a prototypical functional in order to understand better the material response if the misfit (and hence $\theta$) is increased from $0$ to a finite value. We show that for finite, although small, values of $\theta$ there is a unique branch of global minimizers emanating from a spherical cap. For a further study for general values of $\theta \in (0, \infty)$ we then rely on computer experiments. To this end, we develop a penalised, nonconforming finite element discretisation using $P^1$ elements and employ projected gradient descent to solve the ensuing nonlinear problems while ensuring constraints are met. We first show $\Gamma$-convergence of the discrete problems to the continuous one, then investigate the minimizers in their dependence on $\theta$. Interestingly, our results seem to indicate a stark change of material response at a critical value of $\theta$, showing a symmetry breaking `phase transition' from a nearly spherical cap to an approximate cylinder.

\subsection*{Outline}

We begin by recalling our main results from {\cite{DeBenitoSchmidt:19a}} in order to provide the appropriate plate theories in Section~\ref{sec:effective-theories}. There we also identify the effective elastic moduli and spontaneous curvature terms explicitly so as to transform the problem into a more amenable form to identify minimizers. We then discuss the linearised regimes $\alpha \in (2,3)$ and $\alpha > 3$ as well as the asymptotic von K{\'a}rm{\'a}n regimes $\theta \to 0$ and $\theta \to \infty$ in Section~\ref{sec:Lin-regimes}. The structure of minimisers for small $\theta$ is investigated in Section~\ref{sec:structure-minimisers-interpolating}. Finally, Section~\ref{sec:numerics} contains our numerical findings.

%-----------------------------------------------------------------------------------------
%-----------------------------------------------------------------------------------------
%-----------------------------------------------------------------------------------------
\section{Effective plate theories}\label{sec:effective-theories}
%-----------------------------------------------------------------------------------------

We first recall the main results of our contribution \cite{DeBenitoSchmidt:19a} on a hierarchy of plate theories for pre-strained multilayers derived from non-linear three dimensional elasticity by $\Gamma$-convergence. We then determine the effective (homogenised) elastic moduli and corresponding quadratic energy desnities of the plates in terms of the moments of the pointwise elastic constants of the layers.

%-----------------------------------------------------------------------------------------
%-----------------------------------------------------------------------------------------
%-----------------------------------------------------------------------------------------
\subsection{Dimension reduction for pre-strained multilayers}\label{subsec:dim-red}
%-----------------------------------------------------------------------------------------

Working exactly in the setting of \cite{DeBenitoSchmidt:19a} we consider a thin domain
\[ \Omega_h \assign \omega \times (- h / 2, h / 2) \subset \mathbb{R}^3, \]
where $\omega \subset \mathbb{R}^2$ is bounded with Lipschitz boundary, $0 < h \ll 1$, subject to a deformation $w : \Omega_h \to \mathbb{R}^3$. Changing variables form $x_3$ to $x_3 / h$ we obtain a deformation mapping $y(x) = w(x_1,x_2,hx_3)$ and the energy per unit volume 
\[ E^h_{\alpha} (y) 
   = \int_{\Omega_1} W_{\alpha}^h (x_3, \partial_1 y, \partial_2 y, h^{-1} \partial_3 y), \] 
where the elastic energy density $W_{\alpha}^h$ depends on a scaling parameter $\alpha \in (2, \infty)$ and is given by
\begin{equation*}
  W_{\alpha}^h (x_3, F) = W_0 (x_3, F (I +
  h^{\alpha - 1} B^h (x_3))), \applicationspace{1 \tmop{em}} F \in
  \mathbb{R}^{3 \times 3}.
\end{equation*}
for $\alpha \neq 3$, $B^h : \left( - 1 / 2, 1 / 2 \right) \rightarrow
\mathbb{R}^{3 \times 3}$ describing the {\em internal misfit} and $W_0$ the stored energy density of the reference configuration. For $\alpha = 3$ we include
an additional parameter $\theta > 0$ controlling further the amount of misfit
in the model:\label{ref:introducing-theta}
\[ W_{\alpha = 3}^h (x_3, F) = W_0 \big( x_3, F \big( I + h^2  \sqrt{\theta}
   B^h (x_3) \big) \big), \applicationspace{1 \tmop{em}} F \in
   \mathbb{R}^{3 \times 3}. \]
We take $W_0$ fulfilling the usual assumptions of smoothness around
$SO(3)$, frame invariance, boundedness and quadratic growth which are detailed
in \cite{DeBenitoSchmidt:19a}. After linearising around the identity, one obtains the Hessian
\[ Q_3 (t, F) \assign D^2 W_0 (t, I) [F, F] = \frac{\partial^2 W_0 (t,
   I)}{\partial F_{i \nocomma j} \partial F_{i \nocomma j}} F_{i \nocomma j} F
   \nocomma_{i \nocomma j}, \]
for $t \in \left( - 1 / 2, 1 / 2 \right), F \in \mathbb{R}^{3 \times 3}$ and defines $Q_2$ by
minimising away the effect of transversal strain on $Q_3$:
\[ Q_2 (t, G) \assign \underset{c \in \mathbb{R}^3}{\min} Q_3 (t, \hat{G} + c
   \otimes e_3), \]
for $t \in \left( - 1 / 2, 1 / 2 \right), G \in \mathbb{R}^{2 \times 2}$, 
$e_3 = (0, 0, 1) \in \mathbb{R}^3$, and $\hat{G} \in \mathbb{R}^{3 \times 3}$ has $G$ as its upper left $2 \times 2$ submatrix and zeros in the third column and the third row. The functions $Q_2(t, \cdot)$, $t \in (-1/2,1/2)$, are quadratic forms on 
$\mathbb{R}^{2 \times 2}$ which are positive definite on $\mathbb{R}_{\tmop{sym}}^{2 \times 2}$ 
and vanish on antisymmetric matrices. Moreover, they satisfy the bounds 
\begin{equation}\label{eq:Q2-bounds}
  Q_2(t, G) \le C | G |^2 \quad \forall\, G \in \mathbb{R}^{2 \times 2} 
  \quad \text{ and } \quad 
  Q_2(t, G) \ge c | G |^2 \quad \forall\, G \in \mathbb{R}_{\tmop{sym}}^{2 \times 2} 
\end{equation} 
for constants $c, C > 0$ and a.e.\ $t \in (-1/2,1/2)$. We also denote by $\check{B}(t)$ the $2 \times 2$ matrix which arises from $B(t) \in \mathbb{R}^{3 \times 3}$ by deleting its last row and last column. Then 
\begin{equation}\label{eq:B-bounds}
  \check{B} \in L^{\infty} \big( (-1/2,1/2), \mathbb{R}_{\tmop{sym}}^{2 \times 2} \big). 
\end{equation} 
From $Q_2(t, \cdot)$ and $\check{B}(t)$ we define the effective form:
\[ \label{def:Q-bar} \overline{Q}_2 [E, F] \assign \int_{- 1 / 2}^{1 / 2} Q_2
   (t, E + tF + \check{B} (t)) \mathd t, \]
and its relaxation
\begin{equation}
  \label{eq:q2bar} \overline{Q}^{\star}_2 (F) \assign \underset{E \in
  \mathbb{R}_{\tmop{sym}}^{2 \times 2}}{\min}  \int_{- 1 / 2}^{1 / 2} Q_2 (t,
  E + tF + \check{B} (t)) \mathd t.
\end{equation}

In {\cite{DeBenitoSchmidt:19a}} it is shown that $h^{2 - 2\alpha} E^h_{\alpha}$ $\Gamma$-converges for the convergence of the averaged in-plane and out-of-plane displacements $(u^h, v^h) \rightharpoonup (u,v)$ in $W^{1,2}(\omega; \mathbb{R}^3)$ modulo a global rigid motion, cf.\ \eqref{def:scaled-in-out-of-plane-displacements}, to the following effective 
limiting functionals: 

For the scaling $\alpha \in (2, 3)$ as defined in \cite{DeBenitoSchmidt:19a} 
and convex $\omega$, the {\em linearised Kirchhoff} energy is given by
\begin{equation}
  \label{eq:energy-lki} \mathcal{I}_{\rm lKi} (v) \assign
  \left\{\begin{array}{rl}
    \frac{1}{2}  \int_{\omega} \overline{Q}_2^{\star} (- \nabla^2 v) & \text{
    if } v \in W^{2, 2}_{\rm sh} (\omega),\\
    \infty & \text{ otherwise} .
  \end{array}\right.
\end{equation}
For $\alpha = 3$ we have the {\em von K{\'a}rm{\'a}n type
energy}\footnote{As in \cite{DeBenitoSchmidt:19a} we slightly overload
the notation in what would be a double definition of $\mathcal{I}_3^h$, 
using the letter in the subindex to dispel the ambiguity.}
\begin{equation}
  \label{eq:energy-vk} \mathcal{I}^{\theta}_{\rm vK} (u, v) \assign
  \left\{\begin{array}{l}
    \frac{1}{2}  \int_{\omega} \overline{Q}_2 [\theta^{1 / 2} (\grs u +
    \tfrac{1}{2} \nabla v \otimes \nabla v), - \nabla^2 v]\\
    \text{{\hspace{5em}}if } (u, v) \in W^{1, 2} (\omega ; \mathbb{R}^2)
    \times W^{2, 2} (\omega ; \mathbb{R}),\\
    \infty, \text{ otherwise} .
  \end{array}\right.
\end{equation}
Finally, in the regime $\alpha > 3$ we have the {\em linearised von
K{\'a}rm{\'a}n energy}
\begin{equation}
  \label{eq:energy-lvk} \mathcal{I}_{\rm lvK} (u, v) \assign
  \left\{\begin{array}{l}
    \frac{1}{2}  \int_{\omega} \overline{Q}_2 [ \grs u, - \nabla^2 v],\\
    \text{{\hspace{5em}}if } (u, v) \in W^{1, 2} (\omega ; \mathbb{R}^2)
    \times W^{2, 2} (\omega ; \mathbb{R})\\
    \infty, \text{ otherwise} .
  \end{array}\right.
\end{equation}

\begin{remark}\label{rmk:Wh-ass-and-ex}
The precise assumptions on $W^h_{\alpha}$ from \cite{DeBenitoSchmidt:19a} are not essential for the results of the present contribution. In what follows we will only need that the $Q_2(t, \cdot)$, $t \in (-1/2,1/2)$, are quadratic forms on $\mathbb{R}^{2 \times 2}$ that vanish on antisymmetric matrices and satisfy \eqref{eq:Q2-bounds} and that $\check{B}$ satisfies \eqref{eq:B-bounds}. 

The existence of minimizers of \eqref{eq:energy-lki} \eqref{eq:energy-vk} and \eqref{eq:energy-lvk} follows by a standard application of the direct method or, in the setting of \cite{DeBenitoSchmidt:19a}, as a direct consequence of $\Gamma$-convergence and compactness. 
\end{remark}

\noindent{\it Example.} 
For a homogeneous material $Q_2(t,A) = Q_2 (A)$ with linear internal misfit $B (t) = tI$ one has 
\begin{align}
  \mathcal{I}_{\rm lKi} (v) 
  &= \frac{1}{24}  \int_{\omega} Q_2 (\nabla^2 v - I), \nonumber \\ 
  \mathcal{I}_{\rm vK}^{\theta} (u, v) 
  &= \frac{\theta}{2}  \int_{\omega} Q_2 (\grs u + \tfrac{1}{2}
  \nabla v \otimes \nabla v) + \frac{1}{24}  \int_{\omega} Q_2 (\nabla^2 v -
  I), \label{eq:energy-btI} \\ 
  \mathcal{I}_{\rm lvK} (u, v) 
  &= \frac{1}{2} \int_{\omega} Q_2 (\grs u ) + \frac{1}{24}  \int_{\omega} Q_2 (\nabla^2 v - I) \nonumber 
\end{align}
for $v \in  W^{2, 2}_{\rm sh} (\omega)$, respectively, $(u, v) \in W^{1, 2} (\omega ; \mathbb{R}^2) \times W^{2, 2} (\omega ; \mathbb{R})$. These functionals, where the elastic coefficients do not depend on the out-of-plane component, can model for instance a single-layer material under thermal stress. In Section \ref{sec:numerics}, we will study the energy {\eqref{eq:energy-btI}} as a function of $\theta$.

%-----------------------------------------------------------------------------------------
%-----------------------------------------------------------------------------------------
%-----------------------------------------------------------------------------------------
\subsection{Effective moduli and minimising strains}\label{subsec:effective-moduli}
%-----------------------------------------------------------------------------------------

This subsection serves to give explicit formulae relating the homogenised effective elastic moduli found above to the zeroth, first and second moment in $t$ of the individual $Q_2(t, \cdot)$. We also identify their pointwise minimiser so as to rewrite the effective quadratic forms in their most convenient form. The computations are completely elementary, we indicate the main steps. 

Because $Q_2$ vanishes on antisymmetric matrices we may restrict our attention to $F \in \mathbb{R}^{2 \times 2}_{\tmop{sym}}$. From now on, we identify matrices $E = (E_{i \nocomma j})_{i, j = 1}^2 \in \mathbb{R}^{2 \times 2}_{\tmop{sym}}$ with vectors in $\mathbb{R}^3$ via
\begin{equation}\label{eq:identification-matrix-vector} 
  E \mapsto e \assign (E_{11}, E_{22}, E_{12}),
\end{equation}
and analogously $F \mapsto f$, $\check{B} \mapsto b$, $A \mapsto a$. Then, for each $t \in \left( - 1 / 2, 1 / 2 \right)$ there exists some symmetric, positive definite matrix $M (t)$ such that for all $A \in \mathbb{R}^{2 \times 2}_{\tmop{sym}}$: 
\[ Q_2 (t, A) = a^{\top} M (t) a. \]

We define the moments of $M$ as 
\[ M_0 \assign \int_{- 1 / 2}^{1 / 2} M (t) \mathd t, \applicationspace{1
   \tmop{em}} M_1 \assign \int_{- 1 / 2}^{1 / 2} tM (t) \mathd t,
   \applicationspace{1 \tmop{em}} M_2 \assign \int_{- 1 / 2}^{1 / 2} t^2 M
   (t) \mathd t. \]
It is easy to see that \eqref{eq:Q2-bounds} implies that $M_0$ and $M_2$ are positive definite. We claim that also 
\[ M^{\ast} \assign M_2 - M_1 M_0^{- 1} M_1 \] 
is positve definite. To see this, fix $\Lambda \in \mathbb{R}^{2 \times 2}$ and note that for all $x \in  \mathbb{R}^{2} \setminus \{0\}$ 
\begin{align*}
  \int_{- 1 / 2}^{1 / 2} \big| \big( tM^{1 / 2} (t) - M^{1 / 2} (t) \Lambda \big) x \big|^2 \mathd t 
  > 0
\end{align*}
since $\big( tM^{1 / 2} (t) - M^{1 / 2} (t) \Lambda \big) x = 0$ for a.e.\ $t$ would imply that 
$( t I - \Lambda ) x = 0$ in contradiction to $\Lambda$ having at most two eigenvalues. Expanding the square we get  
\begin{align*}
  0 & < 
  %\int_{- 1 / 2}^{1 / 2} x^{\top} 
  %\big( tM^{1 / 2} (t) - M^{1 / 2} (t) \Lambda \big)^{\top} 
  %\big( tM^{1 / 2} (t) - M^{1 / 2} (t) \Lambda \big) x \mathd t\\
  %& = 
  \int_{- 1 / 2}^{1 / 2} x^{\top} (tI - \Lambda)^{\top} M (t)  (tI - \Lambda) \, x \mathd t\\
  & = x^{\top} \big( M_2 - \Lambda^{\top} M_1 - M_1 \Lambda + \Lambda^{\top} M_0 \Lambda \big) \, x
\end{align*}
and, choosing $\Lambda = M_0^{- 1} M_1$, 
\[ 0 < x^{\top} \big( M_2 - M_1 M_0^{- 1} M_1 \big) x. \]

Let $\overline{Q}_2$ be given as in {\eqref{def:Q-bar}}. 
Elementary calculations show that 
\begin{align*}
   &\overline{Q}_2[E, F] 
   ~=~\int_{- 1 / 2}^{1 / 2} Q_2 (t, E + tF + \check{B} (t)) \mathd t \\ 
   &~~=~ e^{\top} M_0 e + f^{\top} M_2 f + \beta_0 + 2 e^{\top} M_1 f + 2 e^{\top} b_1 + 2 f^{\top} b_2 \\ 
%   &~~=~ (M_0 e)^{\top} M_0^{-1} (M_0 e) + f^{\top} M^{\ast} f + (M_1 f)^{\top} M_0^{-1} (M_1 f) + \beta_0 \\ 
%   &~~\qquad + 2 (M_0 e)^{\top} M_0^{-1} M_1 f + 2 (M_0 e)^{\top} M_0^{-1} b_1 + 2 f^{\top} b_2 \\ 
%   &~~=~ \big( M_0 e + M_1 f + b_1 \big)^{\top} M_0^{-1} \big( M_0 e + M_1 f + b_1 \big) - b_1^{\top} M_0^{-1} b_1 \\ 
%   &~~\qquad - 2 (M_1 f)^{\top} M_0^{-1} b_1 + f^{\top} M^{\ast} f + \beta_0 + 2 f^{\top} b_2 \\ 
   &~~=~ \big( e + M_0^{-1} (M_1 f + b_1) \big)^{\top} M_0 \big( M_0 e + M_0^{-1} (M_1 f + b_1) \big) \\ 
   &~~\qquad + \big( f + (M^{\ast})^{-1} (b_2 - M_1 M_0^{-1} b_1) \big)^{\top} M^{\ast} 
               \big( f + (M^{\ast})^{-1} (b_2 - M_1 M_0^{-1} b_1) \big) \\ 
   &~~\qquad - \big( M_1 M_0^{-1} b_1 \big)^{\top} (M^{\ast})^{-1} \big( M_1 M_0^{-1} b_1 \big) 
             - b_1^{\top} M_0^{-1} b_1 + \beta_0 \\  
   &~~=~ \gamma + \big( e + M_0^{-1} (M_1 f + b_1) \big)^{\top} M_0 \big( M_0 e + M_0^{-1} (M_1 f + b_1) \big) \\ 
   &~~\qquad + \big( f + (M^{\ast})^{-1} (b_2 - M_1 M_0^{-1} b_1) \big)^{\top} M^{\ast} 
               \big( f + (M^{\ast})^{-1} (b_2 - M_1 M_0^{-1} b_1) \big), 
\end{align*}
where 
\begin{align}\label{eq:gamma-def} 
  \gamma 
  := - \big( M_1 M_0^{-1} b_1 \big)^{\top} (M^{\ast})^{-1} \big( M_1 M_0^{-1} b_1 \big) 
     - b_1^{\top} M_0^{-1} b_1 + \beta_0. 
\end{align} 

We define the linear mappings $\mathcal{L}_i, \mathcal{L}_{\ast} : \mathbb{R}^{2 \times 2}_{\tmop{sym}} \to \mathbb{R}^{2 \times 2}_{\tmop{sym}}$, $i = 1,2,3$, by 
\[ \mathcal{L}_i A = A' \iff M_i \, a = a', 
   \qquad \text{respectively,} \qquad 
   \mathcal{L}_{\ast} A = A' \iff M^{\ast} \, a = a' \] 
and the positive definite quadratic forms $Q_2^0$ and $Q_2^{\ast}$ on $\mathbb{R}^{2 \times 2}_{\tmop{sym}}$ by 
\begin{align}\label{eq:Q20-Q2ast-def}
  Q_2^0(A) = a^{\top} M_0 \, a, 
  \qquad \text{respectively,} \qquad 
  Q_2^{\ast}(A) = a^{\top} M^{\ast} a. 
\end{align} 
In terms of these quantities our computation reads 
\begin{align}\label{eq:Q2bar}
   \overline{Q}_2[E, F] 
%   = \gamma + Q_2^0 ( E - \mathcal{L}_0^{-1}(\mathcal{L}_1 F + B_1) ) 
%         + Q_2^{\ast} ( F + \mathcal{L}_{\ast}^{-1} (B_2 - \mathcal{L}_1 \mathcal{L}_0^{-1} B_1) ) \\ 
   = \gamma + Q_2^0 ( E - \mathcal{L}_0^{-1}\mathcal{L}_1 F - E_0 ) 
         + Q_2^{\ast} ( F - F_0 ) 
\end{align}
with 
\begin{align}\label{eq:E0-F0-def} 
  F_0 
  = \mathcal{L}_{\ast}^{-1} ( \mathcal{L}_1 \mathcal{L}_0^{-1} B_1 - B_2 ), \ \ 
  E_0 
  = \mathcal{L}_0^{-1} B_1. 
\end{align} 
Minimizing out $E$ yields 
\begin{align}\label{eq:Q2barstar}
  \overline{Q}^{\star}_2 (F) 
  = \underset{E \in \mathbb{R}_{\tmop{sym}}^{2 \times 2}}{\min} \overline{Q}_2[E, F] 
  = \gamma + Q_2^{\ast} ( F - F_0 ). 
\end{align}

%-----------------------------------------------------------------------------------------
%-----------------------------------------------------------------------------------------
%-----------------------------------------------------------------------------------------
\section{Optimal configurations in the linearised and the asymptotic critical regimes}\label{sec:Lin-regimes}
%-----------------------------------------------------------------------------------------

In this section we develop a characterisation of minimisers for the lower range $\alpha \in (2, 3)$ and for the upper range $\alpha > 3$ of scalings. Recall from the discussion in Section \ref{sec:intro} that we are primarily intested in the shape of the out-of-plane component $v$. The results indicate that the characteristic shapes in the limit $h \to 0$ are (infinitesimal) cylinders and paraboloids respectively. Invoking the $\Gamma$-convergence results with respect to the interpolation parameter $\theta$ from \cite[Section 6]{DeBenitoSchmidt:19a} this will also shed light on the optimal shapes in the asymptotic regimes $\theta \to 0$ and $\theta \to \infty$ for the von K{\'a}rm{\'a}n scaling $\alpha = 3$. We collect our results in the following three theorems, where indeed Theorem \ref{thm:lvk-minimizers} is indeed rather an elementary observation based on our preparations form the previous section and Theorem \ref{thm:asymptotic-minimizers} is a direct consequence of \cite[Section 6]{DeBenitoSchmidt:19a}. We allow for a general bounded Lipschitz domain $\omega$ in these theorems. 

\begin{theorem}\label{thm:lvk-minimizers} 
The minimisers of $\mathcal{I}_{\rm lvK}$, eq. \eqref{eq:energy-lvk}, are of the form 
\begin{align}\label{eq:opt-config-lvk}
  u (x) 
  = (\mathcal{L}_0^{-1}\mathcal{L}_1 F_0 + E_0) x 
  \text{\quad and \quad} 
  v (x) = \frac{1}{2} x^{\top} F_0 x, 
\end{align}
with $E_0, F_0 \in \mathbb{R}_{\tmop{sym}}^{2 \times 2}$ the constants from \eqref{eq:E0-F0-def}. $u$ is unique up to an infinitesimal rigid motion and $v$ up to the addition of an affine transformation. 
\end{theorem}

\begin{theorem}\label{thm:lki-minimizers} 
Up to the addition of an affine transformation, the minimisers of $\mathcal{I}_{\rm lKi}$, eq. \eqref{eq:energy-lki}, are of the form 
\begin{align}\label{eq:opt-config-lki}
  v (x) 
  = \frac{1}{2} x^{\top} Fx, \quad 
  F \in \mathcal{N} \assign \tmop{argmin} \big\{ Q_2^{\ast} (F - F_0) : F \in
    \mathbb{R}^{2 \times 2}_{\tmop{sym}},\ \det F = 0 \big\} 
\end{align}
where $Q_2^{\ast}, F_0$ are given in \eqref{eq:Q20-Q2ast-def} and \eqref{eq:E0-F0-def}, respectively. 
\end{theorem}

\begin{remark}
Describing symmetric $2 \times 2$ matrices $A$ by vectors $a \in \mathbb{R}^3$ as in Section \ref{subsec:effective-moduli}, the set $\mathcal{N}$ is the set of touching points of the two quadrics $\{ a \in \mathbb{R}^3 : a_1 a_2 - a_3^2 = 0 \}$ (a cone) and $\{ a \in \mathbb{R}^3 : a^{\top} M^{\ast} a = c_m \}$ (an ellipsoid), where $c_m = Q_2^{\ast}(F - F_0)$ with $F \in \mathcal{N}$. If $\#\mathcal{N} \ge 3$, intersecting with an affine plane $P$ containing three distinct points of $\mathcal{N}$ shows that $\mathcal{N} \cap P$ is an ellipse and then even $\mathcal{N} \subset P$. This shows that either $\#\mathcal{N} = 1$ and there is a unique minimizer, or $\#\mathcal{N} = 2$ and there are precisely two minimizers, or $\mathcal{N}$ is an affine ellipse and to each `winding direction' $\mathbb{R} e$, $e \in S^1$, there is a unique curvature $\lambda = \lambda(e)$ such that $\nabla^2 v \equiv \lambda e \otimes e$. 
\end{remark}

\begin{theorem}\label{thm:asymptotic-minimizers} 
Suppose that $(u^{\theta}, v^{\theta})$ are minimisers of $\mathcal{I}^{\theta}_{\rm vK}$, eq. \eqref{eq:energy-vk}. 
\begin{enumeratealpha}
\item As $\theta \to 0$, up to infinitesimal rigid motions in the in-plane component and up to the addition of affine transformations in the out-of-plane compenent, $(\theta^{1/2} u^{\theta}, v^{\theta}) \rightharpoonup (u, v)$ in $W^{1,2}(\omega,\mathbb{R}^2) \times W^{2,2}(\omega;\mathbb{R})$ with $(u, v)$ as in \eqref{eq:opt-config-lvk}. 

\item As $\theta \to \infty$, up to the addition of affine transformations in the out-of-plane component and up to passing to a subsequence, $v^{\theta} \rightharpoonup v$ in $W^{2,2}(\omega;\mathbb{R})$ with $v$ as in \eqref{eq:opt-config-lki}. 
\end{enumeratealpha}
\end{theorem}

\begin{proof*}{Proof of Theorem \ref{thm:lvk-minimizers}}
By \eqref{eq:energy-lvk} and \eqref{eq:Q2bar}
\begin{align*}
  &\mathcal{I}_{\rm lvK} (u, v) 
    = \frac{1}{2} 
      \int_{\omega} \overline{Q}_2 [ \grs u, - \nabla^2 v ] \mathd x \\
  &~~ = \frac{1}{2} \int_{\omega} \int_{- 1 / 2}^{1 / 2} 
       Q_2^0 ( \grs u + \mathcal{L}_0^{-1}\mathcal{L}_1 \nabla^2 v - E_0 ) 
       + Q_2^{\ast} ( - \nabla^2 v - F_0 ) \mathd x + \frac{\gamma}{2} | \omega |  
\end{align*}
with $u \in W^{1, 2} (\omega ; \mathbb{R}^2)$ and $v \in W^{2, 2} (\omega ; \mathbb{R})$ is minimal (with value $\gamma|\omega| / 2$) if and only if $\nabla^2 v = - F_0$ and $\grs u = \mathcal{L}_0^{-1}\mathcal{L}_1 F_0 + E_0$ a.e. 
\end{proof*}

\begin{proof*}{Proof of Theorem \ref{thm:asymptotic-minimizers}}
a) is immediate from \cite[Theorems 7,10,11]{DeBenitoSchmidt:19a}. b) directly follows from \cite[Theorems 7,8,9]{DeBenitoSchmidt:19a} if $\omega$ is convex. For general $\omega$ first note that the compactness result in \cite[Theorem 7]{DeBenitoSchmidt:19a} does not use convexity, so that $v^{\theta} \rightharpoonup v$ in $W^{2,2}(\omega;\mathbb{R})$ for some $v \in W^{2, 2}_{\rm sh}$. Now fix $F = (f_{ij})_{1\leqslant i,j \leqslant 2} \in \mathcal{N}$ and $\bar{v}(x) = \frac{1}{2} x^{\top} F x$. Since $\det F = 0$, the function $u'(x) = -\frac{1}{3} f_{11} x_1^3 (f_{11}, f_{12}) - f_{12} x_1^2 x_2 (f_{11}, f_{12}) - f_{12} x_1 x_2^2 (f_{12}, f_{22}) - \frac{1}{3} f_{22} x_2^3 (f_{12}, f_{22})$ satisfies $\grs u' + \frac{1}{2} \nabla \bar{v} \otimes \nabla \bar{v} = 0$. Also choose $u''(x) = Ex$ with $E = \mathcal{L}_0^{-1} \mathcal{L}_1 F + E_0$, cf.\ \eqref{eq:Q2bar} and \eqref{eq:E0-F0-def}. Then for $\bar{u} = u' + \theta^{-1/2} u''$ we have by \eqref{eq:Q2barstar}
\begin{align*}
  \mathcal{I}^{\theta}_{\rm vK}(\bar{u},\bar{v}) 
  = \frac{1}{2}  \int_{\omega} \overline{Q}_2 [ \grs u'', - \nabla^2 \bar{v} ] %\\ 
  = \frac{1}{2}  \int_{\omega} \overline{Q}^{\star}_2 (\nabla^2 \bar{v}) 
   = \mathcal{I}_{\rm lKi}(\bar{v}). 
\end{align*}
With the help of the Vitali covering theorem we can exhaust $\omega$ up to a set of negligible measure with disjoint convex subdomains $\omega_1, \omega_2, \ldots$. Denoting the accordingly restricted functionals by $\mathcal{I}^{\theta}_{\rm vK} (\ \cdot\ ; \omega_n)$, $\mathcal{I}_{\rm lKi}(\ \cdot\ ; \omega_n)$ we have 
\begin{align*}
  \tmop{linf}_{\theta \to \infty} \mathcal{I}^{\theta}_{\rm vK}(\bar{u},\bar{v}) 
  &\geqslant \underset{\theta \to \infty}{\tmop{linf}} \mathcal{I}^{\theta}_{\rm vK}(u^{\theta}, v^{\theta}) 
   \geqslant \sum_n \underset{\theta \to \infty}{\tmop{linf}} \mathcal{I}^{\theta}_{\rm vK}(u^{\theta}, v^{\theta}; \omega_n) \\ 
  &\geqslant \sum_n \mathcal{I}_{\rm lKi}(v; \omega_n) 
   \geqslant \sum_n \mathcal{I}_{\rm lKi}(\bar{v}; \omega_n) \\ 
  &= \mathcal{I}_{\rm lKi}(\bar{v}) 
  = \underset{\theta \to \infty}{\tmop{linf}} \mathcal{I}^{\theta}_{\rm vK}(\bar{u},\bar{v}), 
\end{align*}
where we have made use of the lower bound in the $\Gamma$-convegence of $\mathcal{I}^{\theta}_{\rm vK} (\cdot; \omega_n)$ to $\mathcal{I}_{\rm lKi}(\cdot; \omega_n)$, see \cite[Theorem 8]{DeBenitoSchmidt:19a}, in the third step and of Theorem \ref{thm:lki-minimizers} in the fourth step. So we must have $\mathcal{I}_{\rm lKi}(v; \omega_n) = \mathcal{I}_{\rm lKi}(\bar{v}; \omega_n)$ for all $n$ and hence $\nabla^2 v \in \mathcal{N}$ a.e.\ on $\omega$ and so the claim follows from Theorem \ref{thm:lki-minimizers}. 
\end{proof*}

As for Theorem \ref{thm:lki-minimizers}, it is straightforward to see that $v$ as defined in the theorem is a minimisers of $\mathcal{I}_{\rm lKi}$. However, the proof that every minimiser of $\mathcal{I}_{\rm lKi}$ is necessarily of this form needs some work. The difficulty lies in excluding the possibility of constructing a minimiser by piecing together functions whose Hessian belongs to the set $\mathcal{N}$, all with minimal energy but lacking a nice global structure. Yet it is possible to obtain a global representation of the Hessian which shows that it must be constant over $\omega$ so that minimisers are (up to an affine transformation) indeed cylindrical. In order to do this we require (cf.\ {\cite{pakzad_sobolev_2004}}):

\begin{definition}
  \label{def:bodies-arms}Let $\omega' \subset \mathbb{R}^2$ a convex bounded domain and $y \in W^{1, 2}(\omega', \mathbb{R}^3)$ be an isometry. A connected maximal subdomain of $\omega'$ where $\nabla y$ is constant and $y$ is affine whose boundary contains more than two segments inside $\omega'$ is called a {\em body}. A {\em leading curve} is a curve orthogonal to the preimages of $\nabla y$ on the open regions where $\nabla y$ is not constant, parametrised by arc-length. We define an {\em arm} to be a maximal subdomain $\omega (\gamma)$ which is {\em covered} (parametrised) by some leading curve $\gamma$ as follows: 
\[ \omega (\gamma) \subset \{ \phi_{\gamma} (t, s) \assign \gamma (t) + s\nu(t) : s \in \mathbb{R}, t \in [0, l] \}, \]
where $\nu(t) = \gamma' (t)^{\perp}$. We also speak of a {\em covered domain}.
\end{definition}

\begin{figure}[h]
  \centering
  \vspace{-1.2cm}

  \includegraphics{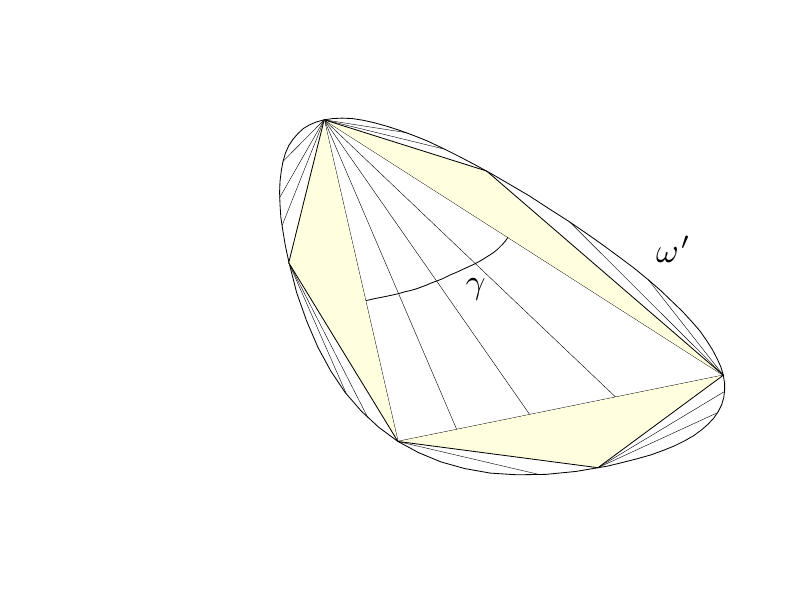}
  \vspace{-1.2cm}

  \caption{The partition of $\omega'$ into bodies and arms. $\nabla y$ is
  constant in the bodies (colored) and along each of the straight lines making
  up the arms (white).}
\end{figure}

The existence of covered domains for isometric immersions $y \in W^{1, 2}$ is
shown in {\cite[Corollary 1.2]{pakzad_sobolev_2004}}.

\begin{proposition}
  \label{thm:repr-hessian}Let $v \in W^{2, 2}_{\rm sh} (\omega)$ and
  $x_0 \in \omega$. There exists a neighbourhood $U$ of $x_0$ such that, if
  $\nabla^2 v \neq 0$ a.e. in $U$, then for a suitable $\varepsilon > 0$ there
  exist maps $\gamma \in W^{2, 2} ((- \varepsilon, \varepsilon) ;
  \mathbb{R}^2)$ and $\lambda \in L^2 ((- \varepsilon, \varepsilon))$ such
  that $U \subset \{ \gamma (t) + s \nu (t) : s \in \mathbb{R}, t \in (-
  \varepsilon, \varepsilon) \}$ and
  \begin{equation}
    \label{eq:thm:repr-matrix:hessian} \nabla^2 v (\gamma (t) + s \nu (t))
    = \frac{\lambda (t)}{1 - s \gamma'' (t)} \gamma' (t) \otimes \gamma' (t)
  \end{equation}
  if $\gamma (t) + s \nu (t) \in U$.
\end{proposition}

\begin{proof}
  We may without loss of generality assume that $\omega$ is convex. 
  Using {\cite[Theorem 10]{friesecke_hierarchy_2006}} take $v_k \in W^{2, 2}
  \cap W^{1, \infty}, S_k \subset \omega$ such that $x_0 \in \tmop{int} S_k$,
  $v_k = v$ on $S_k$ and $\| v_k \|_{1, \infty} \leqslant C$. By scaling $v_k$
  with $\eta > 0$ we can extend $\eta v_k$ to an isometry $y$ ({\cite[Theorem
  7]{friesecke_hierarchy_2006}}) with $\eta v_k = y_3$. Then, because $y$ is
  an isometry:
  \[ - n_3 \tmop{II}_{(y)} = \nabla^2 y_3 = \eta \nabla^2 v \text{\quad on } S_k
      \]
  where $n = y_{,1} \wedge y_{,2}$ is the normal and $\tmop{II}_{(y)} = (\nabla y)^{\top} \nabla n$ the second fundamental form of the surface $y(\omega)$. 
  Since $\nabla^2 y \neq 0$ a.e. near $x_0$, there is a neighbourhood $U$ of
  $x_0$ covered by some leading curve $\gamma$, that is: $U \subset \{ \gamma
  (t) + s \nu (t) : s \in \mathbb{R}, t \in (- \varepsilon, \varepsilon)
  \}$ and, by {\cite[p. 111]{schmidt_plate_2007}}, on $U$ we have
  \[ \tmop{II}_{(y)} (\gamma (t) + s \nu (t)) = \frac{\tilde{\lambda}
     (t)}{1 - s \gamma'' (t)} \gamma' (t) \otimes \gamma' (t), \]
  with $\tilde{\lambda} \in L^2$. Now, {\cite[Proposition 1, eq.
  (12)]{hornung_approximation_2011}} shows that $\nabla y (\gamma (t) + s
  \nu (t))$ is independent of $s$, hence $n_3 = (y_{, 1} \wedge y_{,
  2})_3$ is also independent of $s$ and we can subsume it into the function
  $\tilde{\lambda}$. Setting $\lambda (t) = - n_3 (t)  \tilde{\lambda} (t) /
  \eta$ we obtain the representation {\eqref{eq:thm:repr-matrix:hessian}}.
\end{proof}

Finally, we come to:

\begin{proof*}{Proof of Theorem \ref{thm:lki-minimizers}} 
To recapitulate, according to \eqref{eq:energy-lki} and
\eqref{eq:Q2barstar} the linearised Kirchhoff energy is given by
\begin{equation}
  \label{eq:lki-effective-functional} 
  \mathcal{I}_{\rm lKi}
  (v) = \frac{1}{2}  \int_{\omega} Q_2^{\ast} (\nabla^2 v (x) - F_0) \mathd x +
    \frac{\gamma}{2}  | \omega | 
\end{equation}
for $v \in W^{2, 2}_{\rm sh}$ (and $\infty$ otherwise). 

We observe first that the set $\mathcal{N}= \tmop{argmin} \{ Q_2^{\ast} (F -
  F_0) : F \in \mathbb{R}^{2 \times 2}_{\tmop{sym}}, \det F = 0 \}$ is not
  empty because $F \mapsto Q_2^{\ast} (F - F_0)$ is non-negative and strictly
  convex, but it also need not consist of just one point. Note next that $v$ 
  is a minimiser of 
  {\eqref{eq:lki-effective-functional}} iff $\nabla^2 v (x) \in \mathcal{N}$
  for almost every $x \in \omega$: On the one hand, every minimiser has
  finite energy and thus $\nabla^2 v$ must be pointwise a.e. in the set $\{ F
  \in \mathbb{R}^{2 \times 2}_{\tmop{sym}} : \det F = 0 \}$. On the other, any
  function $F : \omega \rightarrow \mathbb{R}^{2 \times 2}_{\tmop{sym}}$ with
  $F (x) \in \mathcal{N}$ a.e. minimises the integrand in
  {\eqref{eq:lki-effective-functional}} pointwise and thus the energy.
  
  Next we show that any two elements $F, G$ of $\mathcal{N}$ are linearly
  independent. Indeed, by strict convexity we have for all $\lambda \in (0,
  1)$:
  \[ Q_2^{\ast} (\lambda F + (1 - \lambda) G - F_0) < \lambda Q_2^{\ast} (F -
     F_0) + (1 - \lambda) Q_2^{\ast} (G - F_0) . \]
  Hence $\lambda F + (1 - \lambda) G \nin \mathcal{N}$ or else $F, G$ would
  not be minimisers. Because $Q_2^{\ast}$ attains a lower value here we must
  have $\det (\lambda F + (1 - \lambda) G) \neq 0$. But then it cannot be that
  $G = \rho F$ for any scalar $\rho \in \mathbb{R}$ or else it would hold that
  $\det (\lambda F + (1 - \lambda) G) = \det (\lambda F + (1 - \lambda) \rho
  F) = C \det F = 0$, a contradiction. Consequently, we have in particular $0
  \nin \mathcal{N}$ unless $\mathcal{N}= \{ 0 \}$. But in that case $\nabla^2
  v \equiv 0$ and the proof would be concluded.
  
  Let now $v \in W^{2, 2}_{\rm sh}$ be a minimiser for $\mathcal{I}_{\rm lKi}$. 
  Note first that $\nabla v$ cannot be constant over open sets:
  indeed we just saw that w.l.o.g. $0 \nin \mathcal{N}$ and consequently the
  condition $\nabla^2 v = 0$ is excluded for a minimiser on any set of positive 
  measure. Consider then some
  point $x_0 \in \omega$ with a neighbourhood $U$ where $\nabla v$ is not
  constant and use the representation {\eqref{eq:thm:repr-matrix:hessian}}. We
  have that, pointwise a.e.\ and over $U$:
  \[ 0 \neq \nabla^2 v (\gamma(t) + s \nu(t)) = \frac{\lambda (t)}{1 - s \kappa (t)} \gamma'
     (t) \otimes \gamma' (t) . \]
  If $\kappa (t) \neq 0$, by varying $s$ we obtain
  distinct, linearly dependent matrices $\nabla^2 v (t, s)$. 
  Because $\nabla^2 v \in \mathcal{N}$ a.e., this shows that $\kappa(t) = 0$ 
  for a.e.\ $t$. As a consequence, $\gamma'$ must be constant. But then 
  $\lambda$ is also constant or again we would have points at which $\nabla^2
  v$ is linearly dependent. Since this holds locally around every $x = \gamma
  (t) + s \gamma' (t)$, we deduce that $\nabla^2 v$ is constant on $U$ and
  because we can cover $\omega$ in this manner, there exists $F \in 
  \mathcal{N}$ such that $\nabla^2 v \equiv F$ a.e.\ over $\omega$. 
\end{proof*}

%-----------------------------------------------------------------------------------------
%-----------------------------------------------------------------------------------------
%-----------------------------------------------------------------------------------------
\section{Structure of minimisers for 
\texorpdfstring{$\mathcal{I}^{\theta}_{\rm vK}$}{IθvK} for small \texorpdfstring{$\theta$}{θ}}\label{sec:structure-minimisers-interpolating}
%-----------------------------------------------------------------------------------------

The second main contribution of this work is a first study of the properties of minimisers in the interpolating regime, ``close'' to the linearised von K{\'a}rm{\'a}n model. The results in Section \ref{sec:Lin-regimes} show that the transition from spherical to cylindrical shapes occurs in the interpolated von K{\'a}rma{\a'a}n as the strength $\theta$ of the misfit increases. We will see that for small $\theta > 0$ indeed there exists a unique stable branch of solutions emanating from a perfect spherical cap at $\theta = 0$. 

For the sake of clarity we restrict to the prototypical model from \eqref{eq:energy-btI}: 
\[ \mathcal{I}^{\theta}_{\rm vK} (u, v) = \frac{\theta}{2} 
   \int_{\omega} Q_2 (\grs u + \tfrac{1}{2} \nabla v \otimes \nabla v) \mathd
   x + \frac{1}{24} \int_{\omega} Q_2 (\nabla^2 v - I) \mathd x. \]
Natural subsequent steps along this line of work, which we do not take here, are to consider the regime of large values of $\theta$ and to investigate the existence of the conjectured critical value $\theta_c$, as well as to consider the full model derived in \eqref{eq:energy-vk}.\footnote{In Section \ref{sec:numerics} we conduct numerical experiments supporting the conjecture that this critical value exists.}

We recall that the existence of minimizers is guaranteed, cf.\ Remark \ref{rmk:Wh-ass-and-ex}. Without loss of generality wee assume that the barycenter of $\omega$ is $0$. So with $(f)_{\omega} := \frac{1}{|\omega|}\int_{\omega} f (x) \mathd x$ for a function $f$ we in particular have $(x)_{\omega} = 0$. In order to avoid ambiguities (and to apply Korn's and Poincar{\'e}'s inequalities) we restrict the
functions $w = (u, v)$ to lie in the Banach space
\[ X \assign X_u \times X_v, \]
with $X_u, X_v$ as in 
\begin{align*}
  X_u 
  &\assign \left\{ u \in W^{1, 2} (\omega ; \mathbb{R}^2) :
     ( \gra u )_{\omega} = 0 \text{ and } ( u )_{\omega} = 0 \right\}, \\ 
  X_v 
  &\assign \left\{ v \in W^{2, 2} (\omega ; \mathbb{R}) : 
     ( \nabla v )_{\omega} = 0 \text{ and } ( v )_{\omega}= 0 \right\} 
\end{align*}
and norm $\| (u,v) \|_X = ( \| u \|_{1,2}^2 + \| v \|_{2,2}^2 )^{1/2}$. By the arguments in \cite[Remark 2]{DeBenitoSchmidt:19a} working with these spaces 
does not lead to a loss of generality either: For an affine function $g$, $\nabla (v + g) \otimes \nabla (v + g)- \nabla v \otimes \nabla v$ is a symmetrised gradient. 

For small values of the parameter $\theta$ we have the following structural result on the set of minimizers showing the existence of a smooth branch of unique global minimisers. Let $v_0 (x) = \tfrac{1}{2}  | x |^2 - c_0$ with $c_0 = \frac{1}{2}  (| x |^2)_{\omega}$. 
\begin{theorem}
  \label{th:strucure-min-small-theta} There exists an $\varepsilon > 0$, a unique point $u_0 \in X_u$ and a uniquely determined $C^1$ map $\phi : [0, \varepsilon) \rightarrow X$ such that $\phi (0) = (u_0, v_0)$ and for each $\theta \in [0, \varepsilon)$: 
\[ w \in \tmop{argmin} \mathcal{I}^{\theta}_{\rm vK} 
   \quad\iff\quad 
   w = \phi (\theta). \]
\end{theorem}

The proof is a direct consequence of Theorems \ref{th:implicit-function} and \ref{thm:loc-min-are-glob} that are proved in the following two subsections. The main difficulty in obtaining a local branch of minimizers for $\theta \ll 1$ lies in the fact that minimisers at $\theta = 0$ are not unique. Indeed,
\begin{equation}
  \label{eq:structure-minimisers:line-at-v0} (u, v_0) \in \tmop{argmin}
  \mathcal{I}^0_{\rm vK} \text{ for } u \text{ arbitrary},
\end{equation}
as can be readily checked. This is addressed in Subsection~\ref{sec:existence-uniqueness-minimisers-small-theta}. The proof that in fact these minimisers are global is achieved by an application of a Taylor expansion for a carefully perturbed functional in Subsection~\ref{sec:Local-minimisers-are-global}.

%-----------------------------------------------------------------------------------------
\subsection{A branch of solutions for 
\texorpdfstring{$\theta \ll 1$}{θ < < 1}}\label{sec:existence-uniqueness-minimisers-small-theta}

{\begin{notation*}
In this section, the parameter $\theta$ will be explicitly included in the arguments of the functional and differentiation is understood to be with respect to the variables $w = (u, v)$, unless otherwise stated, i.e.
\[ D\mathcal{I}^{\theta}_{\rm vK} (u, v ; \theta) = D_{u, v}
   \mathcal{I}^{\theta}_{\rm vK} (u, v ; \theta) . \]
\end{notation*}}

We are interested in the existence and uniqueness of solutions $w = (u, v)$ to the equation
\[ D\mathcal{I}^{\theta}_{\rm vK} (u, v ; \theta) = 0 \]
as a function of $\theta \in [0, \varepsilon)$ with $\mathcal{I}^{\theta}_{\rm vK}$ given by {\eqref{eq:energy-btI}}. We will in fact prove the existence of a point $(u_0, v_0) \in X$ such that there exists a (locally) unique function $\phi (\theta)$, starting for $\theta = 0$ at $(u_0, v_0)$, such that every $\phi (\theta) \in X$ is a critical point for
$\mathcal{I}^{\theta}_{\rm vK}$. However, lack of uniqueness of minimisers at $\theta = 0$, {\eqref{eq:structure-minimisers:line-at-v0}} will thwart what would be a natural application of the implicit function theorem. The problem manifests itself as a lack of injectivity of the first derivative at $(u, v) \in X$
\begin{align}
  D\mathcal{I}^{\theta}_{\rm vK} (u, v ; \theta)[(\varphi, \psi)] 
  & = \theta \int_{\omega} Q_2 \left[ \grs u +
  \tfrac{1}{2} \nabla v \otimes \nabla v, \grs \varphi + (\nabla v \otimes
  \nabla \psi)_{\tmop{sym}} \right] \nonumber\\
  & \quad + \frac{1}{12}  \int_{\omega} Q_2 [\nabla^2 v - I, \nabla^2
  \psi] .  \label{eq:first-directional-derivative}
\end{align}
which for $\theta = 0$ is
\[ D\mathcal{I}^{\theta}_{\rm vK} (u, v ; 0) [(\varphi, \psi)] = \frac{
   1}{12}  \int_{\omega} Q_2 [\nabla^2 v - I, \nabla^2 \psi], \]
and this vanishes at every $u \in X_u$ and the unique $v (x) = \tfrac{1}{2}  | x |^2 + a \cdot x + b$, 
$a \in \mathbb{R}^2, b \in \mathbb{R}$, such that $( v )_{\omega}  = 0$
and $( \nabla v )_{\omega} = 0$, i.e., $v = v_0$. Because of this the equation
\[ D\mathcal{I}^{\theta}_{\rm vK} (u, v ; \theta) = 0 \text{ in }
   \mathcal{L} (X, \mathbb{R}) \]
cannot be uniquely solvable for $(u, v) \in X$ as a function of $\theta$, even
locally. Nevertheless, after some computations one can see that the problem is the
presence of a leading factor $\theta$ which we can dispense with, because we may apply the implicit
function theorem to the set of equivalent equations

\begin{equation}
\label{eq:unique-minimisers-intro-equivalent-eqs}
( \tfrac{1}{\theta} \partial_u ) \mathcal{I}^{\theta}_{\rm vK} ( u,v; \theta ) = 0,\quad{\partial}_v \mathcal{I}^{\theta}_{\rm vK} (u,v;\theta )=0.
\end{equation}

These equations are equivalent to $D\mathcal{I}^{\theta}_{\rm vK} (u, v
; \theta) = 0$ for any $\theta > 0$ and by an application of the implicit
function theorem around a specific point $(u_0, v_0 ; 0)$ we determine the
existence of a solution function $\phi : \Theta \rightarrow U \times V$ with
$[0, \varepsilon) \subset \Theta, \varepsilon > 0, U \times V \subset X$ open,
$\phi (0) = (u_0, v_0)$ and $\left( \tfrac{1}{\theta} \partial_u, \partial_v
\right) \mathcal{I}^{\theta}_{\rm vK} (\phi (\theta) ; \theta) = 0$.
Then we have $D\mathcal{I}^{\theta}_{\rm vK} (\phi (\theta) ; \theta) =
0$ for $\theta > 0$ because of the equivalence mentioned and
$D\mathcal{I}^{\theta}_{\rm vK} (\phi (0) ; 0) = 0$ by the choice of
$(u_0, v_0)$.

\begin{theorem}
  \label{th:implicit-function} There exists an open set $W$ in $X$, an $\varepsilon > 0$, a point $u_0 \in X_u$ such that $w_0 = (u_0,v_0) \in W$ and a uniquely determined $C^1$ map $\phi : \Theta \rightarrow W$ such that $\phi (0) = w_0$ and
  \[ D\mathcal{I}^{\theta}_{\rm vK} (w ; \theta) = 0 
  \iff w = \phi (\theta) \] 
for all $w \in W$ and $\theta \in [0, \varepsilon)$. 
\end{theorem}

\begin{proof}
  We first define a new set of equations to solve, then show that the second
  derivative of $\mathcal{I}^{\theta}_{\rm vK}$ is one to one and then
  the conclusion is exactly that of the implicit function theorem.
  For brevity we write
  \[ \langle F, G \rangle \assign \int_{\omega} Q_2 [F, G] \text{ and }
     \langle F \rangle \assign \langle F, F \rangle = \int_{\omega} Q_2 (F) .
  \]
  These define a scalar product and a norm in $L^2 (\omega ; \mathbb{R}^{2
  \times 2}_{\tmop{sym}})$ since $Q_2$ is by construction bilinear and symmetric
  and it is positive definite on this space. Even though $Q_2$ vanishes on
  antisymmetric matrices, during the proof we keep track of symmetrised arguments
  to these functions for the sake of clarity.
  
  {\step{Equivalent equations.}{From the computations leading to
  {\eqref{eq:first-directional-derivative}} we have:
  \[ \left( \tfrac{1}{\theta} \partial_u \right) \mathcal{I}^{\theta}_{\rm vK} 
  (u, v ; \theta) [\varphi] = \langle \grs u + \tfrac{1}{2}
     \nabla v \otimes \nabla v, \grs \varphi \rangle, \]
  and
  \begin{align*}
    \partial_v \mathcal{I}^{\theta}_{\rm vK} (u, v ; \theta) [\psi] 
    & = \theta \langle \grs u + \tfrac{1}{2} \nabla v \otimes \nabla v, (\nabla
    v \otimes \nabla \psi)_{\tmop{sym}} \rangle\\
    & \applicationspace{1 \tmop{em}} + \tfrac{1}{12}  \langle \nabla^2 v -
    I, \nabla^2 \psi \rangle 
  \end{align*}
  for all $(\varphi,\psi) \in X$. We observe first that, because 
  $\left( \frac{1}{\theta} \partial_u \right)
  \mathcal{I}^{\theta}_{\rm vK}$ is independent of $\theta$ the right
  hand side makes sense even if $\theta = 0$. Now, on the one hand, for any
  fixed value of $\theta \geqslant 0$ solving the system
  \[ \left\{\begin{array}{rlll}
       \left( \tfrac{1}{\theta} \partial_u \right) \mathcal{I}^{\theta}_{\rm vK} 
       (u, v ; \theta) & = & 0, & \text{in } \mathcal{L} (X_u,
       \mathbb{R}),\\
       \partial_v \mathcal{I}^{\theta}_{\rm vK} (u, v ; \theta) & = & 0,
       & \text{in } \mathcal{L} (X_v, \mathbb{R}),
     \end{array}\right. \]
  implies solving:
  \begin{equation}
    \label{eq:th:implicit-function:equiv-eqs} f (u, v ; \theta) [\varphi,
    \psi] = 0 \text{ for every } (\varphi, \psi) \in X,
  \end{equation}
  where $f : X \times \mathbb{R} \rightarrow \mathcal{L} (X, \mathbb{R})$ is
  given by
  \begin{align*}
    f (u, v ; \theta) [\varphi, \psi] 
    & = \langle \grs u + \tfrac{1}{2}
    \nabla v \otimes \nabla v, \grs \varphi \rangle\\
    &  \quad + \theta \langle \grs u + \tfrac{1}{2} \nabla v \otimes \nabla
    v, (\nabla v \otimes \nabla \psi)_{\tmop{sym}} \rangle\\
    &  \quad + \tfrac{1}{12}  \langle \nabla^2 v - I, \nabla^2 \psi \rangle.
  \end{align*}
  On the other hand, solving $f (u, v ; \theta) = 0$ for $\theta > 0$ is
  equivalent to solving the original problem $D\mathcal{I}^{\theta}_{\rm vK} 
  (u, v ; \theta) = 0$ as we desired.}}
  
  {\step{A zero and the derivative of $f$.}{Since we are interested in the
  behaviour around $\theta = 0$, we evaluate here and obtain
  \[ f (u, v ; 0) [\varphi, \psi] = \langle \grs u + \tfrac{1}{2} \nabla v
     \otimes \nabla v, \grs \varphi \rangle + \tfrac{1}{12}  \langle \nabla^2
     v - I, \nabla^2 \psi \rangle . \]
  We can compute a zero of $f (\cdot, \cdot ; 0)$ by first considering the
  last term, which vanishes for all $\psi \in X_v$ if and only if $v = v_0$. 
  We next observe that the first
  term encodes the orthogonality of $\grs u + \tfrac{1}{2} \nabla v_0 \otimes
  \nabla v_0$ to the space of symmetrised gradients $\tmop{SG}_u \assign \left\{
  \grs \varphi : \varphi \in X_u \right\}$ with respect to the scalar product induced by
  $Q_2$. The $u \in X_u$ realizing this is attained by projecting onto
  $\tmop{SG}_u$, i.e.
  \[ \grs u_0 = - \pi \left( \tfrac{1}{2} \nabla v_0 \otimes \nabla v_0
     \right), \]
  where $\pi : L^2 (\omega ; \mathbb{R}^{2 \times
  2}_{\tmop{sym}}) \rightarrow L^2 (\omega ; \mathbb{R}^{2 \times
  2}_{\tmop{sym}})$ is the orthogonal projection onto $\tmop{SG}_u$ given by
  \[ \pi (B) \assign \underset{A \in \tmop{SG}_u}{\tmop{argmin}} 
     \int_{\omega} Q_2 (B - A) = \underset{A \in \tmop{SG}_u}{\tmop{argmin}} 
     \langle B - A \rangle_{Q_2}. \]
  By the Korn-Poincar{\'e} inequality this determines $u_0 \in X_u$ uniquely. 
  We have then a point $w_0 = (u_0, v_0)$ such that
  \[ f (u_0, v_0 ; 0) = 0 \text{ in } \mathcal{L} (X, \mathbb{R}) . \]
  Finally, we compute $\frac{\mathd}{\mathd \varepsilon} |_{\varepsilon = 0} f
  (u_0 + \varepsilon \varphi_2, v_0 + \varepsilon \psi_2 ; 0) [\varphi_1,
  \psi_1]$ to have the derivative of $f$:
  \begin{align*}
    F (\varphi_2, \psi_2) [\varphi_1, \psi_1] 
    \assign& D_{u, v} f (u_0, v_0
    ; 0) [(\varphi_1, \psi_1), (\varphi_2, \psi_2)]\\
    =& \left\langle \grs \varphi_2, \grs \varphi_1 \right\rangle +
    \left\langle (\nabla v_0 \otimes \nabla
    \psi_2)_{\tmop{sym}}, \grs \varphi_1 \right\rangle\\
    &~+ \tfrac{1}{12}  \langle \nabla^2 \psi_2, \nabla^2 \psi_1
    \rangle.
  \end{align*}}}
  
  {\step{The map $F : X \rightarrow \mathcal{L} (X, \mathbb{R})$ is an
  isomorphism.}{Note first that the map
  \[ \langle (u, v), (\tilde{u}, \tilde{v}) \rangle_X \assign \left\langle
     \grs u, \grs  \tilde{u} \right\rangle + \langle \nabla^2 v, \nabla^2 
     \tilde{v} \rangle \]
  defines a scalar product in $X$, with positive-definiteness following from
  Korn-Poincar{\'e}'s and Poincar{\'e}'s
  inequality. Then we can write $F$ as
  \begin{align*}
    F (\varphi_2, \psi_2) [\varphi_1, \psi_1] 
    & = \langle (\varphi_1, \psi_1), (\varphi_2 + \tilde{\pi} ((\nabla v_0
    \otimes \nabla \psi_2)_{\tmop{sym}}), \tfrac{1}{12} \psi_2) \rangle_X,
  \end{align*}
  where we defined $\tilde{\pi} \assign \grs^{- 1} \circ \pi$, a continuous
  map from $L^2 (\omega ; \mathbb{R}^{2 \times 2}_{\tmop{sym}})$ to $X_u$.
  The Riesz representation for $F (\varphi_2, \psi_2)$ in $\mathcal{L} (X, \mathbb{R})$
  is then $(\varphi_2 + \tilde{\pi} ((\nabla v_0 \otimes \nabla
  \psi_2)_{\tmop{sym}}), \tfrac{1}{12} \psi_2)$ and the map
  \[ (\varphi_2, \psi_2) \mapsto (\varphi_2 + \tilde{\pi} ((\nabla v_0 \otimes
     \nabla \psi_2)_{\tmop{sym}}), \tfrac{1}{12} \psi_2) \]
  is clearly an isomorphism in $X$, with continuity for $\psi_2 \mapsto
  \tilde{\pi} ((\nabla v_0 \otimes \nabla \psi_2)_{\tmop{sym}})$ following from the
  continuity of $\tilde{\pi}$ and the Sobolev embedding $W^{1, 2}
  \hookrightarrow L^4$.}}
\end{proof}

%-----------------------------------------------------------------------------------------
\subsection{Uniqueness and globality of minimisers}\label{sec:Local-minimisers-are-global}

In addition to the previous local result, we can prove that the critical points found in the previous subsection are the unique global minimizers for small non zero values of the parameter $\theta$. We do this in two steps: close to the origin $(u_0, v_0)$ of the branch of solutions, we would like to perform a Taylor expansion and use that the second differential at $(u_0, v_0)$ is ``almost'' positive definite.

The key idea is to slightly modify the energy by a shift and a rescaling in
order to obtain derivatives as those appearing in the equivalent equations
{\eqref{eq:th:implicit-function:equiv-eqs}} of Theorem
\ref{th:implicit-function}, thus obtaining a positive definite second
derivative. We set
\[ \tilde{\mathcal{I}}^{\theta}_{\rm vK} (\tilde{u}, \tilde{v}) \assign
   \mathcal{I}^{\theta}_{\rm vK} \left( u_0 + \tfrac{\tilde{u}}{\theta},
   \tilde{v} \right) \]
and then $(\tilde{u}_{\theta}, \tilde{v}_{\theta})$ is a minimiser of
$\tilde{\mathcal{I}}^{\theta}_{\rm vK}$ if and only if $(u_0 +
\tilde{u}_{\theta} / \theta, \tilde{v}_{\theta})$ is a minimiser of
$\mathcal{I}^{\theta}_{\rm vK}$. In other words, if $(u_{\theta},
v_{\theta})$ is a minimiser of $\mathcal{I}^{\theta}_{\rm vK}$, then
$\tilde{u}_{\theta} = \theta (u_{\theta} - u_0)$ and $\tilde{v}_{\theta} =
v_{\theta}$ minimise $\tilde{\mathcal{I}}^{\theta}_{\rm vK}$.

We name $\tilde{w}_0$ the point around which we investigate the modified
functional:
\begin{equation}
  \label{eq:minimizer-rescaled-shifted-I} \tilde{w}_0 \assign (\tilde{u}_0,
  \tilde{v}_0) = (0, v_0).
\end{equation}
\begin{theorem}
  \label{thm:loc-min-are-glob} There exists $\theta_c > 0$ and a neighborhood 
  $\tilde{W} \subset X$ with $\tilde{w}_0 \in \tilde{W}$ such that for every
  $\theta \in (0, \theta_c)$, every critical point of 
  $D\tilde{\mathcal{I}}^{\theta}_{\rm vK}$ is the unique global minimiser of 
  $\tilde{\mathcal{I}}^{\theta}_{\rm vK}$. 
\end{theorem}

\begin{proof}
  We proceed in three steps. First we prove that there is some $\theta_c > 0$
  such that $D^2 \tilde{\mathcal{I}}^{\theta}_{\rm vK} (\tilde{w})$ is positive
  definite for all $\theta \in (0, \theta_c)$ if $\| \tilde{w} - \tilde{w}_0 \| <
  \eta$ for some suitable $\eta > 0$ and $\tilde{w}_0 = (0, v_0)$ as defined
  in {\eqref{eq:minimizer-rescaled-shifted-I}}. Then we use this to determine
  a neighbourhood of $\tilde{w}_0$ where (local) minimisers of
  $\tilde{\mathcal{I}}^{\theta}_{\rm vK}$ will be global by first
  considering points close to one such minimiser and finally those far away.
  We will need the first two derivatives of $\tilde{\mathcal{I}}^{\theta}_{\rm vK}$.
  
  For the first differential we apply the chain rule to obtain $D_u \nospace
  \tilde{\mathcal{I}}^{\theta}_{\rm vK} (\tilde{u}, \tilde{v}) = \frac{1}{\theta} D_u
  \nospace \mathcal{I}^{\theta}_{\rm vK} \left( u_0 + \frac{\tilde{u}}{\theta},
  \tilde{v} \right)$ and substitute:
  \begin{align*}
    D \nospace \tilde{\mathcal{I}}^{\theta}_{\rm vK} (\tilde{u}, \tilde{v}) [\varphi,
    \psi] & = \langle \grs u_0 + \tfrac{1}{\theta}  \grs \tilde{u} + \tfrac{1}{2}
    \nabla \tilde{v} \otimes \nabla \tilde{v}, \grs \varphi \rangle\\
    &\quad + \theta \langle \grs u_0 + \tfrac{1}{\theta}  \grs \tilde{u} + \tfrac{1}{2}
    \nabla \tilde{v} \otimes \nabla \tilde{v}, (\nabla \tilde{v} \otimes \nabla \psi)_{\tmop{sym}}
    \rangle\\
    &\quad + \tfrac{1}{12}  \langle \nabla^2 \tilde{v} - I, \nabla^2 \psi \rangle .
  \end{align*}
  
  For the second differential we can compute another directional derivative:
  \begin{align}
  &\frac{\mathd}{\mathd \varepsilon} |_{\varepsilon = 0} D \nospace
  \tilde{\mathcal{I}}^{\theta}_{\rm vK} (\tilde{u} + \varepsilon \varphi_2, \tilde{v} +
  \varepsilon \psi_2) [\varphi_1, \psi_1] \nonumber\\ 
  &=~~ \langle \tfrac{1}{\theta} 
  \grs \varphi_2 + (\nabla \tilde{v} \otimes \nabla \psi_2)_{\tmop{sym}}, \grs
  \varphi_1 \rangle \nonumber\\
  &~~\quad + \left\langle \grs \varphi_2 + \theta (\nabla \tilde{v} \otimes \nabla
  \psi_2)_{\tmop{sym}}, \right. \nobracket (\nabla \tilde{v} \otimes \nabla
  \psi_1)_{\tmop{sym}} \rangle \nonumber\\
  &~~\quad + \theta \langle \grs u_0 + \tfrac{1}{\theta}  \grs \tilde{u} + \tfrac{1}{2}
  \nabla \tilde{v} \otimes \nabla \tilde{v}, (\nabla \psi_2 \otimes \nabla
  \psi_1)_{\tmop{sym}} \rangle \nonumber\\
  &~~\quad + \tfrac{1}{12}  \langle \nabla^2 \psi_2, \nabla^2 \psi_1 \rangle . 
  \label{eq:local-min-are-glob-second-deriv}
\end{align}
  
  {\step{Local positive definiteness.}
  {We show there exist $\eta > 0$ and $\theta_c > 0$ s.t. $D^2
  \tilde{\mathcal{I}}^{\theta}_{\rm vK} (\tilde{w})$ is positive definite for
  all $\theta < \theta_c$ and all $\| \tilde{w} - \tilde{w}_0 \|_X < \eta$. More 
  precisely, we even show that there exists some $\bar{c} > 0$ such that 
  \begin{align}\label{eq:HessIvK-posdef} 
  D^2 \tilde{\mathcal{I}}^{\theta}_{\rm vK} (\tilde{w}) [(\varphi, \psi),
     (\varphi, \psi)] \geqslant \bar{c} \| (\varphi, \psi) \|_X^2 
  \end{align} 
  for all $\theta < \theta_c$, $\| \tilde{w} - \tilde{w}_0 \|_X \le \eta$ and 
  $(\varphi, \psi) \in X$. 

  Let then $\eta > 0$ be
  fixed and to be determined later and let $\tilde{w} = (\tilde{u}, \tilde{v}) \in X$ with $\| \tilde{w} -
  \tilde{w}_0 \|_X < \eta$. We start by \ bringing terms together in
  {\eqref{eq:local-min-are-glob-second-deriv}}:
  \begin{align*}
    &D^2 \tilde{\mathcal{I}}^{\theta}_{\rm vK} (\tilde{w})
    [(\varphi, \psi), (\varphi, \psi)] 
    & \\ 
    &~~= \tfrac{1}{\theta}  \left\langle \grs \varphi + \theta
    (\nabla \tilde{v} \otimes \nabla \psi)_{\tmop{sym}} \right\rangle 
    &(a)\\
    &~~\quad + \theta \langle \grs u_0 + \tfrac{1}{\theta}  \grs \tilde{u} + \tfrac{1}{2}
    \nabla \tilde{v} \otimes \nabla \tilde{v}, (\nabla \psi \otimes \nabla \psi)_{\tmop{sym}}
    \rangle 
    &(b)\\
    &~~\quad + \tfrac{1}{12}  \langle \nabla^2 \psi \rangle. 
    &(c)
  \end{align*}
  Given $f, g \in W^{1, 2} (\omega ; \mathbb{R}^2)$ we have, by the bounds
  {\eqref{eq:Q2-bounds}} for $Q_2$ and H{\"o}lder (with the Sobolev embedding
  $W^{1, 2} (\omega ; \mathbb{R}^2) \hookrightarrow L^4 (\omega ;
  \mathbb{R}^2)$):
  \[ \langle (f \otimes g)_{\tmop{sym}} \rangle \lesssim \int_{\omega} | f
     \otimes g |^2 = \int_{\omega} | f |^2  | g |^2 \hspace{0.2em} \leqslant
     \| f \|_{0, 4}^2  \| g \|_{0, 4}^2 \lesssim \| f \|_{1, 2}^2  \| g \|_{1,
     2}^2 . \]
  Using this, the first and last term above can be estimated using
  Korn-Poincar{\'e} and Poincar{\'e}'s inequality:
  \begin{align*}
    (a) & \geqslant \frac{1}{2 \theta}  \left\langle \grs \varphi
    \right\rangle - \theta \langle (\nabla \tilde{v} \otimes \nabla \psi)_{\tmop{sym}}
    \rangle\\
    & \geqslant \frac{c}{2 \theta}  \left\| \grs \varphi \right\|_{0, 2}^2
    - C \theta \| \nabla \tilde{v} \otimes \nabla \psi \|_{0, 2}^2\\
    & \geqslant \frac{c_1}{\theta}  \| \varphi \|_{1, 2}^2 - \tilde{C}_1
    \theta \| \tilde{v} \|_{2, 2}^2  \| \psi \|_{2, 2}^2\\
    & \geqslant c_1 \theta^{- 1} \| \varphi \|_{1, 2}^2 - C_1 \theta \|
    \psi \|_{2, 2}^2
  \end{align*}
  for constants $c_1, C_1, \tilde{C}_1 > 0$, where in the last step we used the assumption 
  $\| \tilde{v} - v_0 \|_{2, 2} < \eta$
  to bound $\| \tilde{v} \|_{2, 2}^2$ by some constant independent of $\eta \leqslant
  1$. For the second term, use Cauchy-Schwarz for $Q_2$, and the same ideas as
  above:
  \begin{align*}
    (b) & \gtrsim - \theta \left\| \grs u_0 + \tfrac{1}{\theta}  \grs \tilde{u} +
    \tfrac{1}{2} \nabla \tilde{v} \otimes \nabla \tilde{v} \right\|_{0, 2}  \| (\nabla \psi
    \otimes \nabla \psi)_{\tmop{sym}} \|_{0, 2}\\
    & \gtrsim - \left[ \theta \left( \left\| \grs u_0 \right\|_{0, 2} + \|
    \nabla \tilde{v} \otimes \nabla \tilde{v} \|_{0, 2} \right) + \left\| \grs \tilde{u} \right\|_{0,
    2} \right]  \| (\nabla \psi \otimes \nabla \psi)_{\tmop{sym}} \|_{0, 2}\\
    & \gtrsim - [\theta (\| u_0 \|_{1, 2} + \| \nabla \tilde{v} \|_{0, 4}^2) + \| \tilde{u}
    \|_{1, 2}]  \| \nabla \psi \|_{0, 4}^2\\
    & \gtrsim - [\theta (\| u_0 \|_{1, 2} + \| \tilde{v} \|_{2, 2}^2) + \eta]  \|
    \psi \|_{2, 2}^2\\
    & \gtrsim - [\theta + \eta]  \| \psi \|_{2, 2}^2 .
  \end{align*}
  Again, we used that by assumption $\| \tilde{u} \|_{1, 2} < \eta$ and $\| \tilde{v} - v_0
  \|_{2, 2} < \eta$.
  
  Finally we estimate the third term in $D^2  \tilde{\mathcal{I}}^{\theta}_{\rm vK}$ 
  with analogous arguments and obtain $(c) \geqslant c_2  \| \psi
  \|_{2, 2}^2$, for a $c_2 > 0$. Bringing the previous computations together, 
  with a $C_2 > 0$ we have:
  \[ D^2 \tilde{\mathcal{I}}^{\theta}_{\rm vK} (\tilde{w}) \geqslant c_1
     \theta^{- 1}  \| \varphi \|_{1, 2}^2 + (c_2 - C_1 \theta - C_2  (\theta +
     \eta))  \| \psi \|_{2, 2}^2, \]
  from which \eqref{eq:HessIvK-posdef} follows if $\theta_c$ and $\eta$ are chosen 
  sufficiently small. 
   
  From now on, we let $\tilde{w}_{\theta} = (\tilde{u}_{\theta},
  \tilde{v}_{\theta})$ be a critical point of
  $\tilde{\mathcal{I}}^{\theta}_{\rm vK}$ with
  \begin{equation}
    \label{eq:thm:loc-min-are-glob:1}
    \|\tilde{w}_{\theta}-\tilde{w}_{0} \|_{X}{\leqslant}{\eta}/3
    \end{equation}
  and we prove that it is in fact the unique global minimizer.}}
  
  \begin{step}{Estimates close to $\tilde{w}_{\theta}$.}
  Consider first some $\tilde{w} \in X$ which is close to $\tilde{w}_{\theta}$:
  \begin{equation}
      \label{eq:thm:loc-min-are-glob:2}
      \| \tilde{w}-\tilde{w}_{\theta} \|_{X} \leqslant 2 \eta/3.
  \end{equation}
  With a Taylor expansion and \eqref{eq:HessIvK-posdef} we see:
  \begin{align*}
     \tilde{\mathcal{I}}^{\theta}_{\rm vK} (\tilde{w}) 
     &= \tilde{\mathcal{I}}^{\theta}_{\rm vK} (\tilde{w}_{\theta}) +
     \underbrace{D \tilde{\mathcal{I}}^{\theta}_{\tmop{vK}}
     (\tilde{w}_{\theta})  [\tilde{w} - \tilde{w}_{\theta}]}_{= 0} + \frac{1}{2} D^2
     \tilde{\mathcal{I}}^{\theta}_{\rm vK} (z)  [\tilde{w} - \tilde{w}_{\theta},
     \tilde{w} - \tilde{w}_{\theta}] \\ 
     &\ge \tilde{\mathcal{I}}^{\theta}_{\rm vK} (\tilde{w}_{\theta})  
     + \frac{\bar{c}}{2} \| \tilde{w} - \tilde{w}_{\theta} \|_X^2, 
  \end{align*}
  where $z \in \{ \alpha \tilde{w} + (1 - \alpha)  \tilde{w}_{\theta} : \alpha \in [0,
  1] \} \subset B_{\eta}(\tilde{w}_0)$ by \eqref{eq:thm:loc-min-are-glob:1} and 
  \eqref{eq:thm:loc-min-are-glob:2}. So 
  \[ \tilde{\mathcal{I}}^{\theta}_{\rm vK} (\tilde{w}) >
     \tilde{\mathcal{I}}^{\theta}_{\rm vK} (\tilde{w}_{\theta}) 
     \text{ unless } \tilde{w} = \tilde{w}_{\theta}. \]
  \end{step}
  
  {\step{Estimates far away from $\tilde{w}_{\theta}$.}
  {Consider now any $\tilde{w} \in X$ with
  \begin{equation}\label{eq:w-faraway}
  \| \tilde{w} - \tilde{w}_{\theta} \|_X > 2 \eta / 3, 
  \end{equation}
  which by \eqref{eq:thm:loc-min-are-glob:1} implies that $\| \tilde{w} - \tilde{w}_0 \|_X > \eta / 3$. We
  consider two cases:
  \smallskip 

  {\noindent}\tmtextit{Case 1: $\| \tilde{v} - v_0 \|_{2, 2} \geqslant \eta / 6$:} We
  discard the first term in the energy, recall that $v_0 (x) = | x |^2 / 2 -
  c_0$ and use the lower bound for $Q_2$ in {\eqref{eq:Q2-bounds}} and
  Poincar{\'e}'s inequality:
  \[ \tilde{\mathcal{I}}^{\theta}_{\rm vK} (\tilde{w}) \geqslant \frac{1}{24} 
     \langle \nabla^2 \tilde{v} - I \rangle \geqslant \frac{c}{24}  \| \nabla^2 (\tilde{v} -
     v_0) \|_{0, 2}^2 \geqslant c_1 \eta^2  \]
  for a $c_1 > 0$. To compare this with the energy at $\tilde{w}_0$ we add and subtract
  $\tilde{\mathcal{I}}^{\theta}_{\rm vK} (\tilde{w}_0) =
  \frac{\theta}{2}  \langle \grs u_0 + \tfrac{1}{2} \nabla v_0 \otimes \nabla
  v_0 \rangle$:
  \begin{align*}
    \tilde{\mathcal{I}}^{\theta}_{\rm vK} (\tilde{w}) & \geqslant 
    \tilde{\mathcal{I}}^{\theta}_{\rm vK} (\tilde{w}_0) + c_1
    \eta^2 - \frac{\theta}{2}  \langle \grs u_0 + \tfrac{1}{2}
    \nabla v_0 \otimes \nabla v_0 \rangle\\
    & >  \tilde{\mathcal{I}}^{\theta}_{\rm vK} (\tilde{w}_0),
    \hspace*{\fill} \text{ for } \theta \text{ small enough,} \hspace*{\fill}\\
    & \geqslant  \tilde{\mathcal{I}}^{\theta}_{\rm vK}
    (\tilde{w}_{\theta}),
  \end{align*}
  where the last line is due to the fact that $\tilde{w}_{\theta}$ minimises
  $\tilde{\mathcal{I}}^{\theta}_{\rm vK}$ over the ball $B_{\frac{2}{3}
  \eta} (\tilde{w}_{\theta})$.
  \smallskip 

  {\noindent}\tmtextit{Case 2: $\| \tilde{v} - v_0 \|_{2, 2} < \eta / 6$:} In this 
  case we also have $\| \tilde{u} \|_{1, 2} \ge \eta / 6$ by 
  \eqref{eq:thm:loc-min-are-glob:1} and \eqref{eq:w-faraway}. We can
  estimate the energy for $\tilde{w}$ as follows:
  \begin{align*}
    \tilde{\mathcal{I}}^{\theta}_{\rm vK} (\tilde{w}) & = \frac{\theta}{2} 
    \langle \grs u_0 + \tfrac{1}{\theta}  \grs \tilde{u} + \tfrac{1}{2} \nabla \tilde{v}
    \otimes \nabla \tilde{v} \rangle + \frac{1}{24}  \langle \nabla^2 \tilde{v} - I \rangle\\
    & \geqslant \frac{1}{2 \theta}  \left\langle \grs \tilde{u} \right\rangle +
    \frac{\theta}{2}  \langle \grs u_0 + \tfrac{1}{2} \nabla \tilde{v} \otimes \nabla
    \tilde{v} \rangle\\
    &  \quad + \langle \grs \tilde{u}, \grs u_0 + \tfrac{1}{2} \nabla \tilde{v} \otimes
    \nabla \tilde{v} \rangle\\
    & \geqslant \Big( \frac{1}{2 \theta} - \varepsilon \Big) 
    \left\langle \grs \tilde{u} \right\rangle + \Big( \frac{\theta}{2} - \frac{1}{4
    \varepsilon} \Big)  \langle \grs u_0 + \tfrac{1}{2} \nabla \tilde{v} \otimes
    \nabla \tilde{v} \rangle\\
    & = \frac{1}{4 \theta}  \left\langle \grs \tilde{u} \right\rangle -
    \frac{\theta}{2}  \langle \grs u_0 + \tfrac{1}{2} \nabla \tilde{v} \otimes \nabla
    \tilde{v} \rangle,
  \end{align*}
  where we used the Cauchy-Schwarz inequality with $\varepsilon \assign \frac{1}{4 \theta}$. Both terms may be
  estimated once again by a combination of the bounds {\eqref{eq:Q2-bounds}}
  for $Q_2$, Sobolev's embedding $W^{1, 2} (\omega) \hookrightarrow L^4
  (\omega)$ and Poincar{\'e}'s inequality:
  \[ \frac{1}{4 \theta}  \left\langle \grs \tilde{u} \right\rangle \gtrsim
     \tfrac{1}{\theta}  \| \tilde{u} \|_{1, 2}^2, \]
  and
  \begin{align*}
    \tfrac{1}{2}  \langle \grs u_0 + \tfrac{1}{2} \nabla \tilde{v} \otimes \nabla \tilde{v}
    \rangle 
    & \leqslant \left\langle \grs u_0 \right\rangle + \tfrac{1}{2} 
    \langle \nabla \tilde{v} \otimes \nabla \tilde{v} \rangle\\
    & \lesssim \left\| \grs u_0 \right\|_{0, 2}^2 + \| \nabla \tilde{v} \|_{0,
    4}^2\\
    & \lesssim 1 + \| \tilde{v} \|_{2, 2}^2 
    \lesssim 1 
  \end{align*}
  since $\| \tilde{v} - v_0 \|_{2, 2}^2 < \eta / 6$. 
  Now plug this back into the previous estimate and insert
  \[ \tilde{\mathcal{I}}^{\theta}_{\rm vK} (\tilde{w}_0) =
     \frac{\theta}{2}  \langle \grs u_0 + \tfrac{1}{2} \nabla v_0 \otimes
     \nabla v_0 \rangle =: \tilde{C} \theta \]
  to obtain
  \begin{align*}
    \tilde{\mathcal{I}}^{\theta}_{\rm vK} (\tilde{w}) & \geqslant 
    \tfrac{c_1}{\theta}  \| \tilde{u} \|_{1, 2}^2 - C_1 \theta (C + C \eta^2 / 9)\\
    &\geqslant \tfrac{c_1 \eta^2}{36 \theta}  \| \tilde{u} \|_{1, 2}^2 - (C_1 + \tilde{C}) \theta 
    + \tilde{\mathcal{I}}^{\theta}_{\rm vK} (\tilde{w}_0)\\
    & > \tilde{\mathcal{I}}^{\theta}_{\rm vK} (\tilde{w}_0),
    \hspace*{\fill} \text{ for } \theta \text{ small enough}, \hspace*{\fill}\\
    & \geqslant \tilde{\mathcal{I}}^{\theta}_{\rm vK}
    (\tilde{w}_{\theta}) .
  \end{align*}
  As above, the last line holds because $\tilde{w}_{\theta}$ minimises
  $\tilde{\mathcal{I}}^{\theta}_{\rm vK}$ in a $\frac{2}{3}
  \eta$-neighbourhood of itself.}}
\end{proof}

\section{Discretisation of the interpolating theory}\label{sec:numerics}

Our goal in this section is to study the qualitative behaviour of minimisers
in the interpolating regime $\alpha = 3$. To this end, we develop a simple 
numerical method to approximate minimisers and prove $\Gamma$-convergence to 
the continuous problem. Numerical computations are then conducted for the 
prototypical example from \eqref{eq:energy-btI}. We experimentally evaluate the
conjectured existence of a critical value $\theta_c > 0$ for which the
symmetry of minimisers is ``strongly'' broken. We will not provide a full
theoretical analysis, but instead adduce some empirical evidence to support
the claim. 

As can only be expected from a topic originating in structural mechanics,
numerical methods for plate models are a vast field with a long history and as
such a comprehensive review falls well beyond the scope of this contribution.
However, it can be said that a significant portion of finite element
approaches focus on the Euler-Lagrange equations. For von K{\'a}rm{\'a}n-like
theories like our interpolating regime, these are transformed into an
equivalent form in terms of the {\tmem{Airy stress function}}
{\cite[{\textsection}2.6.2]{howell_applied_2008}}. The resulting system of
equations is of fourth order and can be solved with conforming $C^1$ elements
like Argyris or specifically taylored ones. To avoid the higher number of
degrees of freedom, non-conforming methods can be used instead,\footnote{See
{\cite{mallik_conforming_2016,mallik_nonconforming_2016}} for particular
instances of a conforming and a non-conforming method respectively, as well as
reviews of recent literature.} but a poor choice of the discretisation can
suffer from {\tmem{locking}}, as briefly described in Remark
\ref{rem:fem-common-issues}. Some successful classical methods employ $C^0$
Discrete Kirchhoff triangles (DKT), but it is also possible to
employ standard Lagrange elements with penalty methods
{\cite{brenner_interior_2017}}, as we will do.

A recent line of work, upon which we heavily build in this section, is that of
{\cite{bartels_approximation_2013,bartels_numerical_2016}}, where the author
develops discrete gradient flows for the direct computation of (local)
minimisers of non-linear Kirchhoff and von K{\'a}rm{\'a}n models.
$\Gamma$-convergence and compactness results are also proved showing the
convergence of the discrete energies to the continuous ones, as well as their
respective minimisers.\footnote{For a concise introduction to
$\Gamma$-convergence for Galerkin discretisations and quadrature
approximations of energy functionals, see {\cite{ortner_gamma_2004}}.}
Crucially, these papers use DKTs for the discretisation of the out-of-plane
displacements, allowing for a representation of derivatives at nodes in the
mesh which is decoupled from function values. This enables e.g.\ the imposition
of an isometry constraint for the non-linear Kirchhoff model, but also the
computation of a discrete gradient $\nabla_{\varepsilon}$ projecting the true
gradient $\nabla v_{\varepsilon}$ of a discrete function $v_{\varepsilon}$
into a standard piecewise $P_2$ space. The operator $\nabla_{\varepsilon}$ has
good interpolation properties circumventing the lack of $C^1$ smoothness of
DKTs which would otherwise make them unsuitable to approximate solutions in
$H^2$. We refer to the book {\cite{bartels_numerical_2015}} for a systematic and
mostly self-contained introduction to these methods.

\subsection{Discretisation}

We wish to investigate minimal energy configurations of the following
functional:
\[ \mathcal{I}^{\theta}_{\tmop{vK}} (u, v) 
   = \frac{1}{2}  \int_{\omega} \overline{Q}_2 \big[ \theta^{1 / 2} (\grs u +
    \tfrac{1}{2} \nabla v \otimes \nabla v), - \nabla^2 v \big] \mathd x, \]
where $(u, v) \in W^{1, 2} (\omega ; \mathbb{R}^2) \times W^{2, 2} (\omega ;
\mathbb{R}^2)$, cf.\ \eqref{eq:energy-vk}. We recall the representation of 
$\overline{Q}_2$ derived in \eqref{eq:Q2bar}, which in particular shows that 
$\overline{Q}_2$ is a strictly convex polynomial of degree $2$ on 
$\mathbb{R}^{2 \times 2}_{\tmop{sym}} \times \mathbb{R}^{2 \times 2}_{\tmop{sym}}$. 
It is extended to a convex quadratic function on 
$\mathbb{R}^{2 \times 2} \times \mathbb{R}^{2 \times 2}$ by our setting 
\[ \overline{Q}[E, F] = \overline{Q}[E_{\tmop{sym}}, F_{\tmop{sym}}] \] 
for $F, G \in \mathbb{R}^{2 \times 2}$. We assume that $\omega \subset \mathbb{R}$
is a bounded simply connected domain with Lipschitz  boundary and barycenter $0$. 
We implement (projected) gradient descent in a non-conforming
method using $C^0$ linear Lagrange elements. The first step is to transform the
problem into one of constrained minimisation reducing the order of the
elements required. 

\begin{problem}
  \label{prob:cont-constrained-vk}Find minimisers of
  \begin{equation}
    \label{eq:vk-curl-energy} J^{\theta} (u, z) 
    = \frac{1}{2}  \int_{\omega} \overline{Q}_2 \big[ \theta^{1 / 2} (\grs u +
    \tfrac{1}{2} z \otimes z), - \nabla z \big] \mathd x,
  \end{equation}
  with $u, z \in
  W^{1, 2} (\omega ; \mathbb{R}^2)$ and
  \[ z \in Z \assign \{ \zeta \in W^{1, 2} (\omega ; \mathbb{R}^2) :
     \tmop{curl} \zeta = 0 \} . \]
  If $z \nin Z$, then we set $J^{\theta} (u, z) = + \infty$.
\end{problem}

Note that our assumptions on $\omega$ guarantee that 
$Z = \{ \nabla v : v \in W^{2, 2} (\omega) \}$. 
We can now use $H^1$-conforming elements but, for simplicity of
implementation, instead of adding the constraint into the discrete spaces to
obtain a truly conforming discretisation, we add a penalty term
$\mu_{\varepsilon}  \| \tmop{curl} z_{\varepsilon} \|^2$ to ensure that the
solutions $z_{\varepsilon}$ are close to gradients.

Assume from now on that $\omega$ is a polygonal domain. For fixed $\varepsilon
> 0$, introduce a {\em quasi-uniform} triangulation
$\mathcal{T}_{\varepsilon}$ of $\omega$ with triangles $T$ of uniformly
bounded diameter \ $c^{- 1} \varepsilon \leqslant \varepsilon_T \leqslant c
\varepsilon$ for some $c > 0$ and all $\varepsilon > 0$ and $T \in
\mathcal{T}_{\varepsilon}$.\footnote{ Note that this does not allow for
arbitrary local refinements or {\em grading} (a different scaling of
simplices along different directions as $\varepsilon \rightarrow 0$), but the
fact that this is not optimal is not of concern here.} Such a mesh is in
particular said to be, in virtue of the uniform upper bound,
{\em shape-regular}. We denote by $\mathcal{N}_{\varepsilon}$ the set of
all nodes of the triangulation. Define $V_{\varepsilon}$ to be the standard
piecewise affine, globally continuous Lagrange $P_1$ finite element space
$\mathcal{S}^1 (\mathcal{T}_{\varepsilon})$ in two dimensions:
\[ V_{\varepsilon} \assign \left\{ v_{\varepsilon} \in C (\overline{\omega} ;
   \mathbb{R}^2) : v_{\varepsilon |T} \in P_1 (T)^2 \text{ for all } T \in
   \mathcal{T}_{\varepsilon} \right\} . \]
Quadrature rules will be chosen to be exact for this polynomial degree and the
first integrand in the energy interpolated for this to apply by means of the
{\em interpolated quadratic function}
\[ \overline{Q}^{\varepsilon}_2 \assign \hat{I}_{\varepsilon} \circ \overline{Q}_2. \]
This is defined (with a slight abuse of notation) component-wise using the
{\em element-wise nodal interpolant} $\hat{I}_{\varepsilon}$, defined for
functions $v \in L^{\infty} (\omega)$ such that $v_{|T} \in C (\overline{T})$
for all $T \in \mathcal{T}_{\varepsilon}$ as
\begin{equation}
  \label{eq:elementwise-nodal-interpolant} \hat{I}_{\varepsilon} (v) \assign
  \sum_{T \in \mathcal{T}_{\varepsilon}} \sum_{z \in \mathcal{N}_{\varepsilon}
  \cap T} v_{|T} (z) \varphi_{z|T},
\end{equation}
where $\varphi_{z|T}$ is the truncation by zero outside $T$ of the global
basis function $\varphi_z \in \mathcal{S}^1$. Because this is a linear
combination of truncated global basis functions, the range of
$\hat{I}_{\varepsilon}$ is the space $\hat{\mathcal{S}}^1
(\mathcal{T}_{\varepsilon})$ of discontinuous, piecewise affine Lagrange
elements.

In cases where the function to be interpolated is continuous, the
element-wise nodal interpolant coincides with the {\em standard nodal
interpolant} into the space $\mathcal{S}^1$ of globally continuous, piecewise
affine functions, which is defined as
\begin{equation}
  \label{eq:nodal-interpolant} I_{\varepsilon} (v) \assign 
    \sum_{z \in \mathcal{N}_{\varepsilon}} v (z) \varphi_z.
\end{equation}

Notice that the shape functions $\varphi_z$ are not truncated. In order to
control the error incurred by the interpolation. When working with discontinuous functions
in $\hat{\mathcal{S}}^1$, we will \ use the following local result.
This follows from standard nodal interpolation estimates (see e.g.
{\cite[Theorem 4.28]{grossmann_numerical_2007}} or
{\cite[(4.4.4)]{brenner_mathematical_2008}})
\[ | I_{\varepsilon} (v) - v |_{r, p} \leqslant C \varepsilon^{2 - r}  \|
     D^2 v \|_{0, p}, \]
 or can be shown directly, e.g. in
{\cite[Proposition 3.1]{bartels_numerical_2015}}.

\begin{lemma}[Local interpolation estimate]
  \label{lem:local-interpolation-estimate}Let $T \in
  \mathcal{T}_{\varepsilon}$ and $v \in C^1 (\overline{T})$. If
  $\hat{I}_{\varepsilon}$ is the element-wise nodal interpolant
  {\eqref{eq:elementwise-nodal-interpolant}}, then
  \[ \| v - \hat{I}_{\varepsilon} (v) \|_{0, p, T} \leqslant C \varepsilon \|
     D v \|_{0, p, T} . \]
\end{lemma}

The goal is to solve:

\begin{problem}
  \label{prob:discrete-vk}Let $\mu_{\varepsilon} > 0$. Compute minimisers of
  the discrete energy
  \begin{align}
  \begin{split}
    J^{\theta}_{\varepsilon} (u_{\varepsilon}, z_{\varepsilon}) & =
    \frac{1}{2}  \int_{\omega} \overline{Q}^{\varepsilon}_2 \big[ \theta^{1 / 2} (\grs u_{\varepsilon} +
    \tfrac{1}{2} z_{\varepsilon} \otimes z_{\varepsilon}), - \nabla z_{\varepsilon} \big] \mathd x \\
    &  \qquad + \mu_{\varepsilon}  \int_{\omega} |
    \tmop{curl} z_{\varepsilon} |^2 \mathd x, 
    \label{eq:discrete-vk-curl-energy}
  \end{split}
  \end{align}
  for $(u_{\varepsilon}, z_{\varepsilon}) \in V_{\varepsilon}^2$.
(As usual, if $(u_{\varepsilon}, z_{\varepsilon}) \in W^{1, 2} (\omega ;
\mathbb{R}^2)^2 \backslash V^2_{\varepsilon}$, we set
$J^{\theta}_{\varepsilon} (u_{\varepsilon}, z_{\varepsilon}) = + \infty$.) 
\end{problem}

\begin{remark}[Scaling of the constants]
  The penalty $\mu_{\varepsilon} = \mu (\varepsilon)$ needs to explode as
  $\varepsilon \rightarrow 0$ in order for the functionals to
  $\Gamma$-converge (Theorem \ref{thm:discrete-vk-gamma-convergence}).
  However, large penalties negatively affect the condition number of the
  system, so that an adequate choice for $\mu_{\varepsilon}$, dependent on the
  mesh size $\varepsilon$, is required
  {\cite[p.416]{grossmann_numerical_2007}}. We have not explictly
  investigated how this requirement interacts with the $\Gamma$-convergence of
  the functionals, but in our proof we require only that $\mu_{\varepsilon}
  \rightarrow \infty$ not faster than $\varepsilon^{- 2}$. In the
  implementation we use $\mu_{\varepsilon} = \varepsilon^{-1/2}$.
  Analogously, large values of the Lam{\'e} constants have a similar effect
  and therefore hinder convergence, so one needs to scale them to the order of
  the problem.
\end{remark}

%\begin{remark}[Automatic fulfilment of the constraint] 
%  Experiments seemed to indicate that under some circumstances, in particular
%  not too unfavourable initialisations, one can set $\mu = 0$ and still obtain
%  minimisers with vanishing $\tmop{curl}$. However, if this is done, when
%  $\theta$ is increased and approaches the (conjectured) critical value, long
%  energy plateaus are traversed after which large, markedly non physical
%  deformations take place. Because the constraint is part of the
%  discretisation, we did not further investigate this phenomenon.
%\end{remark}

\begin{remark}[Common issues with FEM for plates]
  \label{rem:fem-common-issues} Discretisations for lower dimensional theories
  can face complications due to the infamous {\em locking phenomena}. In a
  nutshell, these mean that as the thickness of the plate tends to zero,
  discrete solutions ``lock'' to stiff states of lower, or even vanishing,
  bending or shearing than the analytic ones.\footnote{We refer to
  {\cite{babuska_locking_1992}} for a first rigorous definition of locking, to
  {\cite[Chapters 5 and 6]{prathap_finite_2001}} for detailed computations
  highlighting the issues with linear elements in the context of Timoshenko
  beams and to the thesis {\cite{quaglino_membrane_2012}} for a thorough and
  detailed analysis of locking in shell models.} Another instance of
  unexpected behaviour is known as the Babu{\v s}ka paradox
  {\cite{babuska_plate_1990}}, again a failure to converge as expected, which
  can happen in e.g. the Kirchhoff model when both vertical and tangential
  displacements are fixed at the boundaries of a polygonal domain: these
  so-called ``hard'' support constraints are not enforced in the same manner
  as in the continuous model because of the approximated domain.
  
  There are two potential sources of locking in our setting: the penalty term
  $\mu_{\epsilon}$, which is akin to the shear strain in Timoshenko beams, and
  $\theta$. We have not obtained any a priori bounds on the error in this work,
  but a rigorous treatment of the problem would require estimates which are uniform
  in these parameters as the mesh diameter goes to zero. For the regimes studied
  and the geometries considered we have found the issue to be of moderate 
  practical relevance, but it does manifest itself e.g. with more complicated
  domains or higher values of $\theta$.
  
  Finally, our simulations will not suffer from Babu{\v s}ka's paradox because
  we do not prescribe boundary conditions.
\end{remark}

\subsection{\texorpdfstring{$\Gamma$}{Γ}-convergence of the discrete energies}

The first step in the proof that $J^{\theta}_{\varepsilon}
\overset{\Gamma}{\rightarrow} J^{\theta}$ is dispensing with the interpolation
operators for numerical integration: due to the good properties of
$\hat{I}_{\varepsilon}$, we can assume that we work with the true integrals
$\int \overline{Q}_2$ instead of $\int \overline{Q}_2^{\varepsilon}$: 

\begin{lemma}[Numerical integration]
  \label{lem:numerical-integration}Let $u_{\varepsilon}, z_{\varepsilon} \in
  W^{1, 2} (\omega ; \mathbb{R}^2)$ be uniformly bounded in $W^{1, 2}$ and let
  $Q_2^{\varepsilon} = \hat{I}_{\varepsilon} \circ Q_2$ as above. Let
  $A_{\varepsilon} \assign \big( \theta^{1 / 2} ( \grs u_{\varepsilon} + \tfrac{1}{2} z_{\varepsilon}
  \otimes z_{\varepsilon}), - \nabla z_{\varepsilon} \big)$. Then, as $\varepsilon \rightarrow 0$:
  \[ \| \overline{Q}_2^{\varepsilon} [A_{\varepsilon}] - \overline{Q}_2 [A_{\varepsilon}] \|_{0, 1}
     \rightarrow 0. \]
\end{lemma}

\begin{proof}
  By the local interpolation estimate Lemma
  \ref{lem:local-interpolation-estimate}:
  \begin{align*}
    &\int_{\omega} | \overline{Q}_2^{\varepsilon} [A_{\varepsilon}] - \overline{Q}_2
    [A_{\varepsilon}] | \mathd x \\ 
    &~~ \lesssim \varepsilon \sum_{T \in
    \mathcal{T}_{\varepsilon}} \int_T | D \overline{Q}_2 [A_{\varepsilon}] | \mathd x\\
    &~~ \lesssim \varepsilon \sum_{T \in
    \mathcal{T}_{\varepsilon}} \int_T (1 + | A_{\varepsilon} |)  | D
    A_{\varepsilon} | \mathd x\\
    &~~ \lesssim \varepsilon \left( \sum_{T \in \mathcal{T}_{\varepsilon}}
    \int_T (1 + | A_{\varepsilon} |)^2 \mathd x \right)^{1 / 2}  \left( \sum_{T \in
    \mathcal{T}_{\varepsilon}} \int_T | D A_{\varepsilon} |^2 \mathd
    x \right)^{1 / 2}.
  \end{align*}
  Now, the first term is simply $\| 1 + |A_{\varepsilon}| \|_{0, 2, \omega} \le 
  |\omega|^{1/2} + \| A_{\varepsilon} \|_{0, 2, \omega}$ which
  is uniformly bounded since $\| z_{\varepsilon} \otimes z_{\varepsilon}
  \|_{0, 2} = \| z_{\varepsilon} \|_{0, 4}^2 \lesssim \| z_{\varepsilon}
  \|_{1, 2}^2$, and for the second we use that both $\grs u_{\varepsilon}$ 
  and $\nabla z_{\varepsilon}$ are piecewise constant so that for $i = 1, 2$,
  \[ | \partial_i A_{\varepsilon} |^2 = \theta | z_{\varepsilon} \otimes \partial_i
     z_{\varepsilon} + \partial_i z_{\varepsilon} \otimes z_{\varepsilon} |^2
     \lesssim | z_{\varepsilon} |^2  | \partial_i z_{\varepsilon} |^2, 
     \]
  and
  \[ \sum_{T \in \mathcal{T}_{\varepsilon}} \int_T | \partial_i
     A_{\varepsilon} |^2 \mathd x \lesssim \sum_{T \in
     \mathcal{T}_{\varepsilon}} \int_T | z_{\varepsilon} |^2  | \partial_i
     z_{\varepsilon} |^2 \mathd x \leqslant \sum_{T \in
     \mathcal{T}_{\varepsilon}} \| z_{\varepsilon} \|_{0, \infty, T}^2  \|
     \partial_i z_{\varepsilon} \|_{0, 2, T}^2 . \]
  A standard inverse estimate (see e.g. {\cite[Theorem
  4.5.11]{brenner_mathematical_2008}}) provides the bound
  \[ \underset{T \in \mathcal{T}_{\varepsilon}}{\max}  \| z_{\varepsilon}
     \|_{0, \infty, T} \lesssim \varepsilon^{- 1 / 2}  \left( \sum_{T \in
     \mathcal{T}_{\varepsilon}} \| z_{\varepsilon} \|_{0, 4, T}^4 \right)^{1 /
     4} . \]
  We plug this into the preceding computation to obtain
  \begin{align*}
    \sum_{T \in \mathcal{T}_{\varepsilon}} \int_T | \partial_i A_{\varepsilon}
    |^2 \mathd x & \lesssim \varepsilon^{- 1}  \left( \sum_{T \in
    \mathcal{T}_{\varepsilon}} \| z_{\varepsilon} \|_{0, 4, T}^4 \right)^{1 /
    2}  \sum_{T \in \mathcal{T}_{\varepsilon}} \| \partial_i z_{\varepsilon}
    \|_{0, 2, T}^2\\
    & = \varepsilon^{- 1}  \| z_{\varepsilon} \|_{0, 4, \omega}^2  \|
    \partial_i z_{\varepsilon} \|^2_{0, 2, \omega} .
  \end{align*}
  The last two norms being uniformly bounded, we conclude:
  \[ \int_{\omega} | \overline{Q}_2^{\varepsilon} [A_{\varepsilon}] - \overline{Q}_2
     [A_{\varepsilon}] | \mathd x \lesssim \sum_{i = 1}^2 \sum_{T \in
     \mathcal{T}_{\varepsilon}} \varepsilon \int_T | \partial_i \overline{Q}_2
     [A_{\varepsilon}] | \mathd x \lesssim \varepsilon^{1 / 2} \rightarrow 0.
  \]
\end{proof}

The second step is, as usual, to ensure that we can focus on smooth functions
for simplicity in the construction of the upper bound:

\begin{lemma}
  \label{lem:curl-free-smooth-functions-dense} The
  set $C^{\infty} (\overline{\omega}, \mathbb{R}^2) \cap Z$ is $W^{1, 2}$-dense in 
  $Z$.
\end{lemma}

\begin{proof}
  This follows from $Z = \{ \nabla v : v \in W^{2, 2} (\omega) \}$ and the density 
  of $C^{\infty} (\overline{\omega})$ in $W^{2, 2} (\omega)$. 
\end{proof}

\begin{theorem}
  \label{thm:discrete-vk-gamma-convergence}Let $J^{\theta},
  J^{\theta}_{\varepsilon}$ be given by {\eqref{eq:vk-curl-energy}} and
  {\eqref{eq:discrete-vk-curl-energy}} respectively. Assume that 
  $\mu_{\varepsilon} \to \infty$ such that 
  $\mu_{\varepsilon} = o (\varepsilon^{- 2})$ as $\varepsilon \rightarrow 0$.
  Then $J^{\theta}_{\varepsilon} \overset{\Gamma}{\rightarrow} J^{\theta}$ as
  $\varepsilon \rightarrow 0$ with respect to weak convergence in $W^{1, 2}$.
\end{theorem}

\begin{proof}
  Because of Lemma \ref{lem:numerical-integration} we can substitute $\overline{Q}_2$ for
  $\overline{Q}_2^{\varepsilon}$ in $J^{\theta}_{\varepsilon}$. Also, by Lemma
  \ref{lem:curl-free-smooth-functions-dense}
  it is enough to consider smooth functions for the upper bound. Set
  \[ A \assign \big( \theta^{1 / 2} ( \grs u + \tfrac{1}{2} z \otimes z ), - \nabla z \big) \text{\quad and\quad}
     A_{\varepsilon} \assign \big( \theta^{1 / 2} ( \grs u_{\varepsilon} + \tfrac{1}{2}
     z_{\varepsilon} \otimes z_{\varepsilon} ) , - \nabla z_{\varepsilon} \big). \]
  {\step{Upper bound.}{Let $(u, z) \in W^{1, 2} (\omega ; \mathbb{R}^2) \times
  Z$ be $C^{\infty}$ up to the boundary and define $u_{\varepsilon} \assign
  I_{\varepsilon} (u), z_{\varepsilon} \assign I_{\varepsilon} (z)$, where
  $I_{\varepsilon}$ is the nodal interpolant of
  {\eqref{eq:nodal-interpolant}}. Note that because $u$ and $z$ are smooth, we
  can apply standard interpolation estimates to show strong convergence in $W^{1, 2}$ of
  these sequences towards $u$ and $z$. By the compact Sobolev embedding $W^{1, 2}
  \hookrightarrow L^4$ we have $z_{\varepsilon} \rightarrow z$ in $L^4$, and
  $z_{\varepsilon} \otimes z_{\varepsilon} \rightarrow z \otimes z$ in $L^2$,
  so we have that $A_{\varepsilon} \rightarrow A$ in $L^2$. Since $\overline{Q}_2$ 
  is a polynomial of degree 2, this implies 
  \begin{align*}
    \int_{\omega} \overline{Q}_2 [A_{\varepsilon}] \mathd x 
    \to \int_{\omega} \overline{Q}_2 [A] \mathd x
  \end{align*}
  as $\varepsilon \rightarrow 0$. 
  By the same interpolation estimate above and the assumption on
  $\mu_{\varepsilon}$ we have that $\mu_{\varepsilon}  \| \tmop{curl}
  (\hat{I}_{\varepsilon} (z) - z) \|^2_{0, 2} = o (1)$ as $\varepsilon
  \rightarrow 0$, and consequently
  \[ J^{\theta}_{\varepsilon} (u_{\varepsilon}, z_{\varepsilon}) \rightarrow
     J^{\theta} (u, z) . \]}}
  
  {\step{Lower bound.}{Let $u_{\varepsilon}, z_{\varepsilon} \in
  V_{\varepsilon} \subset W^{1, 2}$ with $u_{\varepsilon} \rightharpoonup u$, and
  $z_{\varepsilon} \rightharpoonup z$ weakly in $W^{1, 2}$ to $u \in W^{1, 2}
  (\omega ; \mathbb{R}^2), z \in Z$. Because
  $z_{\varepsilon} \otimes z_{\varepsilon} \rightarrow z \otimes z$ in $L^2$,
  we have that $A_{\varepsilon} \rightharpoonup A$ in $L^2$. Moreover, 
  $\tmop{curl} z_{\varepsilon} \rightharpoonup \tmop{curl} z$. If 
  $\underset{\varepsilon \rightarrow 0}{\tmop{linf}} J^{\theta}_{\varepsilon} = 
  \infty$, the assertion is trivial. If not, then $\mu_{\varepsilon_k} 
  \int_{\omega} | \tmop{curl} z_{\varepsilon_k} |^2 \mathd x \leqslant C$ 
  and $\| \tmop{curl} z_{\varepsilon_k} \|_{0, 2} \rightarrow 0$ for a 
  subsequence $\varepsilon_k \rightarrow 0$. But then $\tmop{curl} z = 0$. 
  Dropping the (non-negative) $\tmop{curl}$ term in $J^{\theta}_{\varepsilon}$ and by the
  weak sequential lower semicontinuity of all integrands involved ($\overline{Q}_2$ 
  being a convex quadratic function), we then get 
  \begin{align*}
    \underset{\varepsilon \rightarrow 0}{\tmop{linf}} J^{\theta}_{\varepsilon}
    (u_{\varepsilon}, z_{\varepsilon}) 
    \geqslant 
    \int_{\omega} \overline{Q}_2 [A] \mathd x 
    = J^{\theta} (u, z).
  \end{align*}}}
\end{proof}

The final ingredient of this subsection is a proof that sequences with bounded
energy are (weakly) precompact. The fundamental theorem of
$\Gamma$-convergence then shows convergence of global minimisers.
In order for this to work, we need to assume
conditions in the space which provide Korn and Poincar{\'e} inequalities. We
can do this using functions with zero mean, zero mean of the gradient or zero
mean of the antisymmetric gradient as we do above,
but including these conditions in the discrete spaces is not entirely trivial.
Because the energies are invariant under the transformations which are
factored out by taking quotient spaces as described in the sections mentioned,
it is enough for our purposes to claim compactness modulo these
transformations and to exclude them in the implementation via projected
gradient descent.

\begin{theorem}[Compactness]
  \label{thm:discrete-vk-compactness}Let $(u_{\varepsilon},
  z_{\varepsilon})_{\varepsilon > 0}$ be a sequence in $(V_{\varepsilon} \cap
  X_u)^2$ with bounded energy. Then there exist $u \in W^{1, 2}, z \in Z$ such
  that $u_{\varepsilon} \rightharpoonup u$ and $z_{\varepsilon}
  \rightharpoonup z$. in $W^{1, 2}$.
\end{theorem}

\begin{proof}
  As above, let $A_{\varepsilon} \assign \big( \theta^{1 / 2} ( \grs u_{\varepsilon} 
  + \frac{1}{2} z_{\varepsilon} \otimes z_{\varepsilon}), - \nabla z_{\varepsilon} \big)$. 
  Note that we cannot use Lemma \ref{lem:numerical-integration}
  to substitute $Q_2$ for $Q_2^{\varepsilon}$ since we do not have uniform
  bounds in $W^{1, 2}$ by assumption, so we work directly with
  $J^{\theta}_{\varepsilon}$.
  
  We begin by observing that, as $\overline{Q}_2 : \mathbb{R}^{2 \times 2} \times 
  \mathbb{R}^{2 \times 2}_{\tmop{sym}} \to \mathbb{R}$ is a convex quadratic function bounded 
  from below which is strictly convex on $\mathbb{R}^{2 \times 2}_{\tmop{sym}} 
  \times \mathbb{R}^{2 \times 2}_{\tmop{sym}}$, there are constants $\bar{c}, \bar{C} > 0$ such that 
  \[ \overline{Q}_2[E, F] \ge \bar{c}|E_{\tmop{sym}}|^2 + \bar{c}|F_{\tmop{sym}}|^2 - \bar{C} \] 
  for all $E, F \in \mathbb{R}^{2 \times 2}$. In particular, 
  \[ \bar{c} \| \nabla z_{\varepsilon} \|_{0, 2}^2 
     \leqslant J^{\theta}_{\varepsilon}(u_{\varepsilon}, z_{\varepsilon}) 
     + \bar{C} |\omega|, \]
  and consequently, by Poincar{\'e}'s inequality:
  \begin{equation}
  \label{eq:compactness-fem-z}
  \| z_{\varepsilon} \|_{1,2} \leqslant C.
  \end{equation}
  
  We have then a subsequence (not relabeled) weakly converging in $W^{1, 2}$
  to some $z \in W^{1, 2}$. In particular $\nabla z_{\varepsilon}
  \rightharpoonup \nabla z$ and $\tmop{curl} z_{\varepsilon} \rightharpoonup
  \tmop{curl} z$ in $L^2$. But also
  \[ \mu_{\varepsilon}  \| \tmop{curl} z_{\varepsilon} \|_{0, 2}^2 \leqslant
     C \Rightarrow \tmop{curl} z_{\varepsilon} \rightarrow 0 \text{ in } L^2,
  \]
  and therefore $\tmop{curl} z = 0$, i.e.\ $z \in Z$.
  
  Now, for the sequence $u_{\varepsilon}$ we must work with
  $\overline{Q}_2^{\varepsilon}$ instead. First write
  \begin{align*}
    \bar{c} | \theta^{1 / 2} \grs u_{\varepsilon} |^2 
    &\leqslant 2 \bar{c} | \theta^{1 / 2} ( \grs u_{\varepsilon}
         + \tfrac{1}{2} z_{\varepsilon} \otimes z_{\varepsilon} ) |^2 
         + 2 \bar{c} | \theta^{1 / 2} \tfrac{1}{2} z_{\varepsilon} \otimes z_{\varepsilon} |^2 \\ 
    &\leqslant 2 \overline{Q}_2[A_{\varepsilon}] + 2 \bar{C} 
         + \tfrac{1}{2} \bar{c} \theta | z_{\varepsilon} \otimes z_{\varepsilon} |^2 
  \end{align*}
  and thus 
  \begin{align*}
    | \grs u_{\varepsilon} |^2 
    \lesssim \overline{Q}_2[A_{\varepsilon}] + | z_{\varepsilon} |^4 + C. 
  \end{align*}
  Since this applies pointwise, after (local) interpolation the estimate still
  holds:
  \[ | \grs u_{\varepsilon} | 
     = \hat{I}_{\varepsilon} | \grs u_{\varepsilon} | 
     \lesssim \overline{Q}^{\varepsilon}_2 [A_{\varepsilon}] + \hat{I}_{\varepsilon} (|
     z_{\varepsilon} |^4) + C, \]
  where in the firs step we have used that $\grs u_{\varepsilon}$ is piecewise constant. So 
  \[ \| \grs u_{\varepsilon} \|_{0,2} 
     \lesssim J^{\theta}_{\varepsilon}
     (u_{\varepsilon}, z_{\varepsilon}) + \int_{\omega} \hat{I}_{\varepsilon}
     (| z_{\varepsilon} |^4) \mathd x + C. \]

  We claim now that $\| \hat{I}_{\varepsilon} (| z_{\varepsilon} |^4) - |
  z_{\varepsilon} |^4 \|_{0, 1} =\mathcal{O} (\varepsilon)$. Indeed, by the
  local interpolation estimate (Lemma \ref{lem:local-interpolation-estimate})
  and H{\"o}lder's inequality for integrals and for sums:
  \begin{align*}
    \int_{\omega} | \hat{I}_{\varepsilon} (| z_{\varepsilon} |^4) - |
    z_{\varepsilon} |^4 | & \lesssim \varepsilon \sum_{T \in
    \mathcal{T}_{\varepsilon}} \int_T | \nabla | z_{\varepsilon} |^4 |\\
    & \lesssim \varepsilon \sum_{T \in \mathcal{T}_{\varepsilon}} \int_T |
    z_{\varepsilon} |^3  | \nabla z_{\varepsilon} |\\
    & \lesssim \varepsilon \sum_{T \in \mathcal{T}_{\varepsilon}} \|
    z_{\varepsilon} \|_{0, 6, T}^3  \| \nabla z_{\varepsilon} \|_{0, 2, T}\\
    & \lesssim \varepsilon \left( \sum_{T \in \mathcal{T}_{\varepsilon}} \|
    z_{\varepsilon} \|_{0, 6, T}^6 \right)^{1 / 2}  \left( \sum_{T \in
    \mathcal{T}_{\varepsilon}} \| \nabla z_{\varepsilon} \|^2_{0, 2, T}
    \right)^{1 / 2}\\
    & \lesssim \varepsilon \| z_{\varepsilon} \|^3_{0, 6, \omega}  \|
    \nabla z_{\varepsilon} \|_{0, 2, \omega},
  \end{align*}
  and this goes to zero as $\varepsilon \rightarrow 0$ by
  {\eqref{eq:compactness-fem-z}}. But then $\int_{\omega}
  \hat{I}_{\varepsilon} (| z_{\varepsilon} |^4) \leqslant C$ and by 
  Korn-Poincar{\'e}'s inequality, the Sobolev embedding $W^{1, 2} \hookrightarrow
  L^4$ and the previous bound, we have
  \[ \| u_{\varepsilon} \|^2_{1, 2} \lesssim \left\| \grs u_{\varepsilon}
     \right\|^2_{0, 2} 
     \lesssim J^{\theta}_{\varepsilon} (u_{\varepsilon}, z_{\varepsilon}) + C
     \leqslant C. \]
  The sequence $(u_{\varepsilon})_{\varepsilon > 0}$ is therefore also weakly
  precompact in $W^{1, 2} (\omega ; \mathbb{R}^2)$ and the proof is complete.
\end{proof}

\subsection{Discrete gradient flow}

As a concrete example we specialize now to the prototypical example 
\[ \mathcal{I}^{\theta}_{\rm vK} (u, v) = \frac{\theta}{2} 
   \int_{\omega} Q_2 (\grs u + \tfrac{1}{2} \nabla v \otimes \nabla v) \mathd
   x + \frac{1}{24} \int_{\omega} Q_2 (\nabla^2 v - I) \mathd x, \]
cf.\ \eqref{eq:energy-btI}.
For each discrete problem, we compute local minimisers using gradient descent,
for which the basic result is the following (see
{\cite[{\textsection}4.3.1]{bartels_numerical_2015}}):

\begin{theorem}[Projected gradient descent]
  Let $V_{\varepsilon}$ and $J^{\theta}_{\varepsilon}$ be given as in Problem
  \ref{prob:discrete-vk} and let $(\cdot, \cdot)$ be the scalar product on
  $V_{\varepsilon}$. The map $F_{\varepsilon} : V_{\varepsilon} \times
  V_{\varepsilon} \rightarrow (V_{\varepsilon} \times V_{\varepsilon})'$ given
  by
  \begin{align}
    F^{\theta}_{\varepsilon} [u_{\varepsilon}, z_{\varepsilon}]
    (\varphi_{\varepsilon}, \psi_{\varepsilon}) & \assign  \theta
    \int_{\omega} Q^{\varepsilon}_2 [\grs u_{\varepsilon} + \tfrac{1}{2}
    z_{\varepsilon} \otimes z_{\varepsilon}, \grs \varphi_{\varepsilon} +
    (z_{\varepsilon} \otimes \psi)_s] \mathd x \nonumber\\
    &  \qquad + \frac{1}{12}  \int_{\omega} Q_2 [\nabla z_{\varepsilon} - I,
    \nabla \psi_{\varepsilon}] \mathd x \nonumber\\
    &  \qquad + 2 \mu_{\varepsilon}  \int_{\omega} \tmop{curl}
    z_{\varepsilon} \tmop{curl} \psi_{\varepsilon} \mathd x, 
    \label{eq:discrete-vk-curl-gradient}
  \end{align}
  is the Fr{\'e}chet derivative of $J^{\theta}_{\varepsilon}$. Let $\pi_u :
  V_{\varepsilon}^2 \rightarrow (V_{\varepsilon} \cap X_u)^2$ be the linear
  orthogonal projection onto its image. The sequence defined as
  \[ w^{j + 1}_{\varepsilon} \assign w^j_{\varepsilon} + \alpha_j \pi_u
     d_{\varepsilon}^j \]
  with $w_{\varepsilon}^0 = (u_{\varepsilon}^0, v_{\varepsilon}^0) \in
  (V_{\varepsilon} \cap X_u)^2$ and $d_{\varepsilon}^j \in V_{\varepsilon}
  \times V_{\varepsilon}$ such that
  \begin{equation}
    \label{eq:discrete-gradient-descent} (d^j_{\varepsilon},
    \xi_{\varepsilon}) = - F^{\theta}_{\varepsilon} [w^j_{\varepsilon}]
    (\xi_{\varepsilon}) \text{ for all } \xi_{\varepsilon} \in V_{\varepsilon}
    \times V_{\varepsilon},
  \end{equation}
  and $\alpha_j$ determined with line search is energy decreasing. A line
  search means computing the maximal $\alpha_j \in \{ 2^{- k} : k \in
  \mathbb{N} \}$ such that
  \[ J^{\theta}_{\varepsilon} (w^j_{\varepsilon} + \alpha_j \pi_u
     d_{\varepsilon}^j) \leqslant J^{\theta}_{\varepsilon} (w^j_{\varepsilon})
     - \rho \alpha_j  \| \pi_u d^j_{\varepsilon} \|^2_2, \]
  where $\rho \in (0, 1 / 2)$ is the proverbial fudge factor.
\end{theorem}

\begin{proof}
  The computation of $F^{\theta}_{\varepsilon}$ is straightforward.
  To see that the iteration is energy decreasing use
  {\eqref{eq:discrete-gradient-descent}} and the self-adjointness of $\pi_u =
  \pi^2_u$ to compute
  \[ \left. \frac{\mathd}{\mathd \alpha} \right|_{\alpha = 0} J^{\theta}_{\varepsilon}
     (w^j_{\varepsilon} + \alpha \pi_u d_{\varepsilon}^j) =
     F^{\theta}_{\varepsilon} [w^j_{\varepsilon}] (\pi_u d_{\varepsilon}^j) =
     - (\pi_u d^j_{\varepsilon}, \pi_u d^j_{\varepsilon}) \leqslant 0. \]
  The existence of $\alpha_j > 0$ is guaranteed as long as
  $J^{\theta}_{\varepsilon} \in C^2 (V_{\varepsilon}^2)$ because then we can
  perform a Taylor expansion and use again
  {\eqref{eq:discrete-gradient-descent}}:
  \[ J^{\theta}_{\varepsilon} (w^j_{\varepsilon} + \alpha_j \pi_u
     d_{\varepsilon}^j) = J^{\theta}_{\varepsilon} (w^j_{\varepsilon}) -
     \alpha_j  \| \pi_u d^j_{\varepsilon} \|^2_{\mathcal{S}} +\mathcal{O}
     (\alpha_j^2) . \]
\end{proof}

\begin{remark}[Caveat: local and global minimisers]
  Even though we now know that the discrete energies correctly approximate
  the continuous one, as well as any global minimisers, gradient descent on
  each discrete problem is only guaranteed to converge to some local minimiser
  $w^{\star}_{\varepsilon}$. Lacking some means of tracking a particular
  $w^{\star}_{\varepsilon}$ as $\varepsilon \rightarrow 0$, there is not much
  one can do to prove that our method actually approximates the true global
  minimisers of $\mathcal{I}_{\tmop{vK}}^{\theta}$. Unless $\theta \ll 1$, in
  which case we know local minimisers to be global (cf. Theorem
  \ref{thm:loc-min-are-glob}).
\end{remark}

\subsection{Experimental results}

For the implementation of the discretisation detailed above, we employ the
{\tmname{FEniCS}} library {\cite{alnaes_fenics_2015}} in its version 2017.1.0.
The code is available at {\cite{debenitodelgado_implementation_2017a}} and
includes the model, parallel execution, experiment tracking using
{\tmname{Sacred}} {\cite{greff_sacred_2017}} with {\tmname{MongoDB}} as a
backend and exploration of results with {\tmname{Jupyter}}
{\cite{kluyver_jupyter_2016}} notebooks, {\tmname{Omniboard}}
{\cite{subramanian_webbased_2018}} and a custom application. Everything is
packaged using {\tmname{docker-compose}} for simple reproduction of the
results and one-line deployment.

We set $\omega = \hat{B}_1 (0)$, a (coarse) polygonal approximation of the
unit disc and test several initial conditions. The space $V_{\varepsilon}$ has
$\sim$7000 dofs. We implement a general $Q_2$ for isotropic homogeneous material
with the two (scaled) Lam{\'e} constants set to those of steel at standard
conditions. We apply neither body forces nor boundary conditions, but hold one
interior cell to fix the value of the free constants. We compute minimisers
for increasing values of $\theta$ and $\mu_{\varepsilon} \sim 1 /
\sqrt{\varepsilon}$ via projected gradient descent (onto the space of
admissible functions $V_{\varepsilon} \cap X_u$) and examine the symmetry of
the final solution. The choice $\varepsilon^{- 1 / 2}$ has shown to provide
the fastest convergence results while keeping the violation of the constraint
in the order of $10^{- 4}$ (higher penalties have the expected effect of
adversely affecting convergence). We track two magnitudes as measures of
symmetry: on the one hand we compute the mean bending strain over the domain
and on the other, as a second simple proxy we employ the quotient of the
lengths of the principal axes.

The first initial configuration is the trivial deformation $y^0_{\varepsilon}
= 0$. Note that because the model is prestrained, the ground state is
non-trivial and the plate ``wants'' to reach a lower energy state. In Figure
\ref{fig:minimisers-zero} we depict the results of running the energy
minimisation procedure for multiple values of $\theta$.

\begin{figure}[h]
  \centering
  \begin{tabular}{ll}
    \resizebox{0.25\columnwidth}{!}{\includegraphics{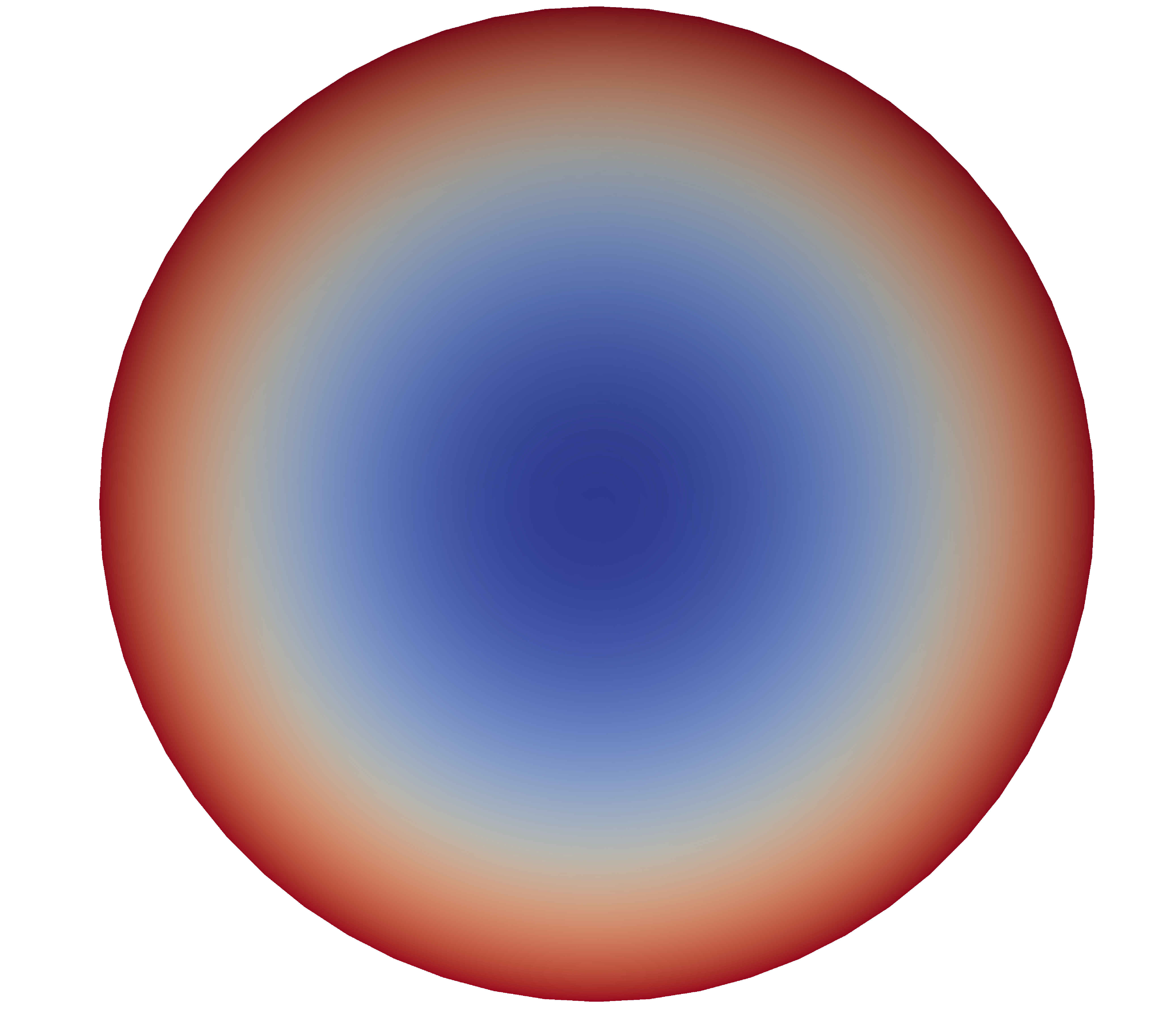}}
    &
    \resizebox{0.25\columnwidth}{!}{\includegraphics{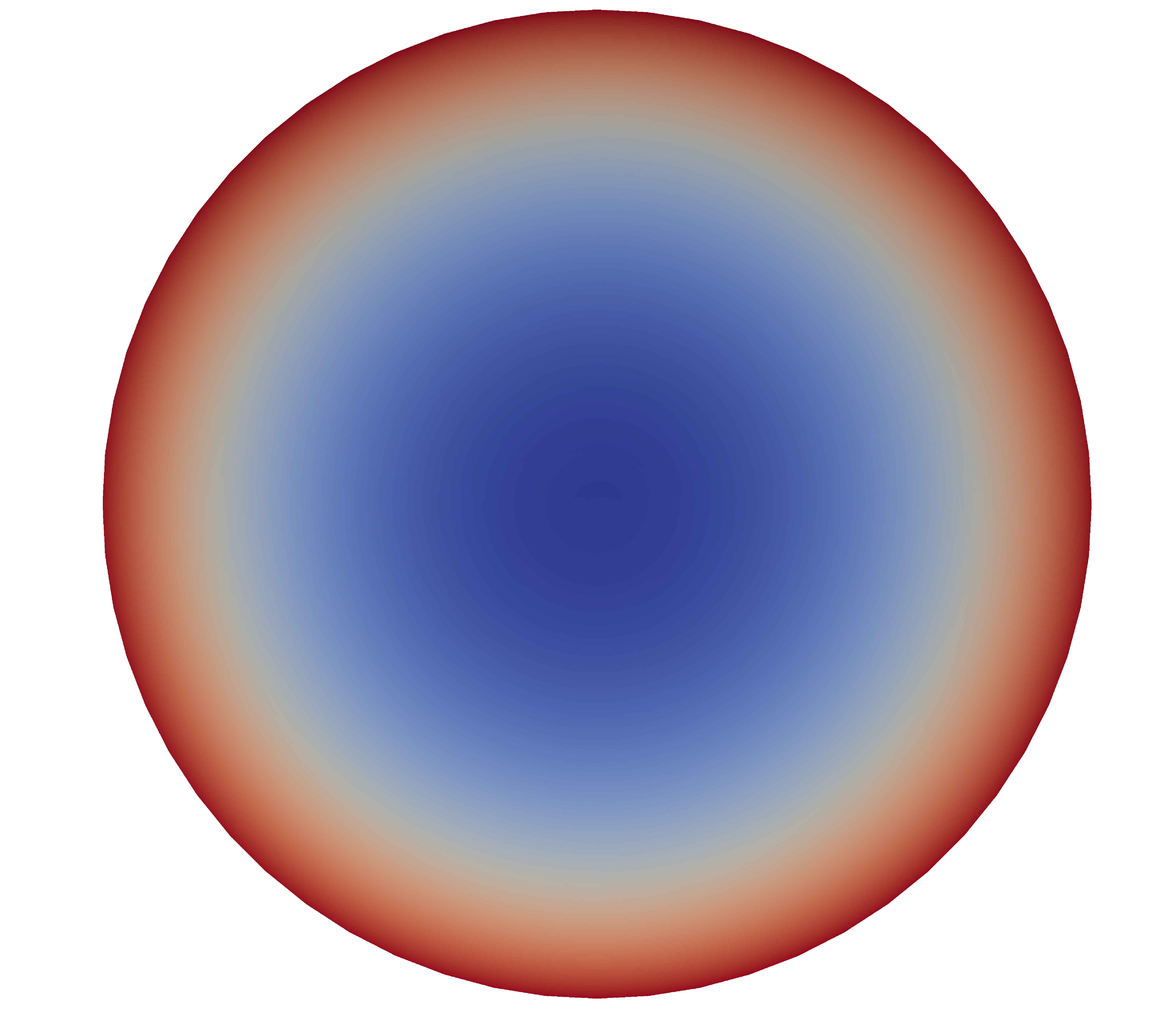}}\\
    \resizebox{0.25\columnwidth}{!}{\includegraphics{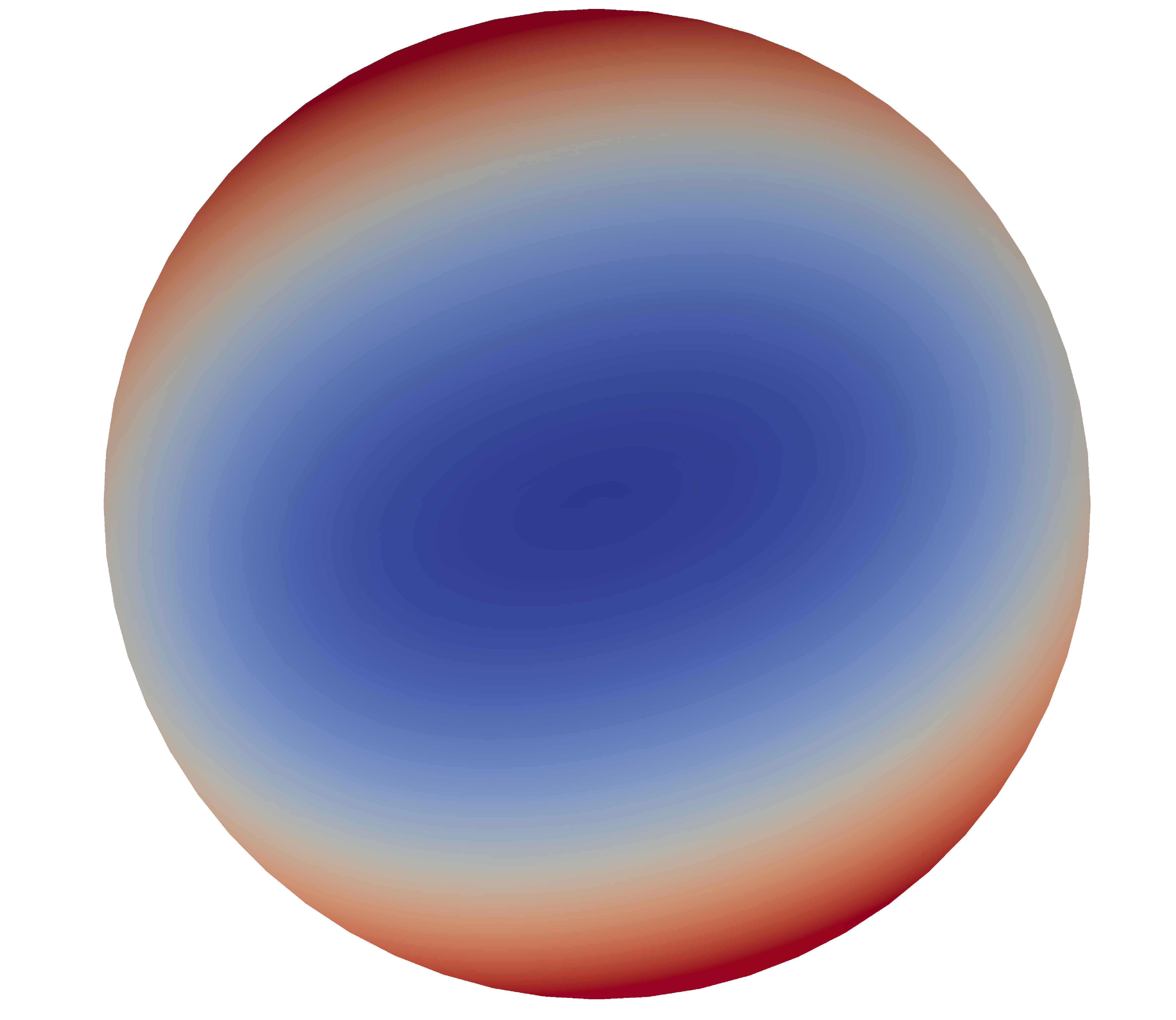}}
    &
    \resizebox{0.25\columnwidth}{!}{\includegraphics{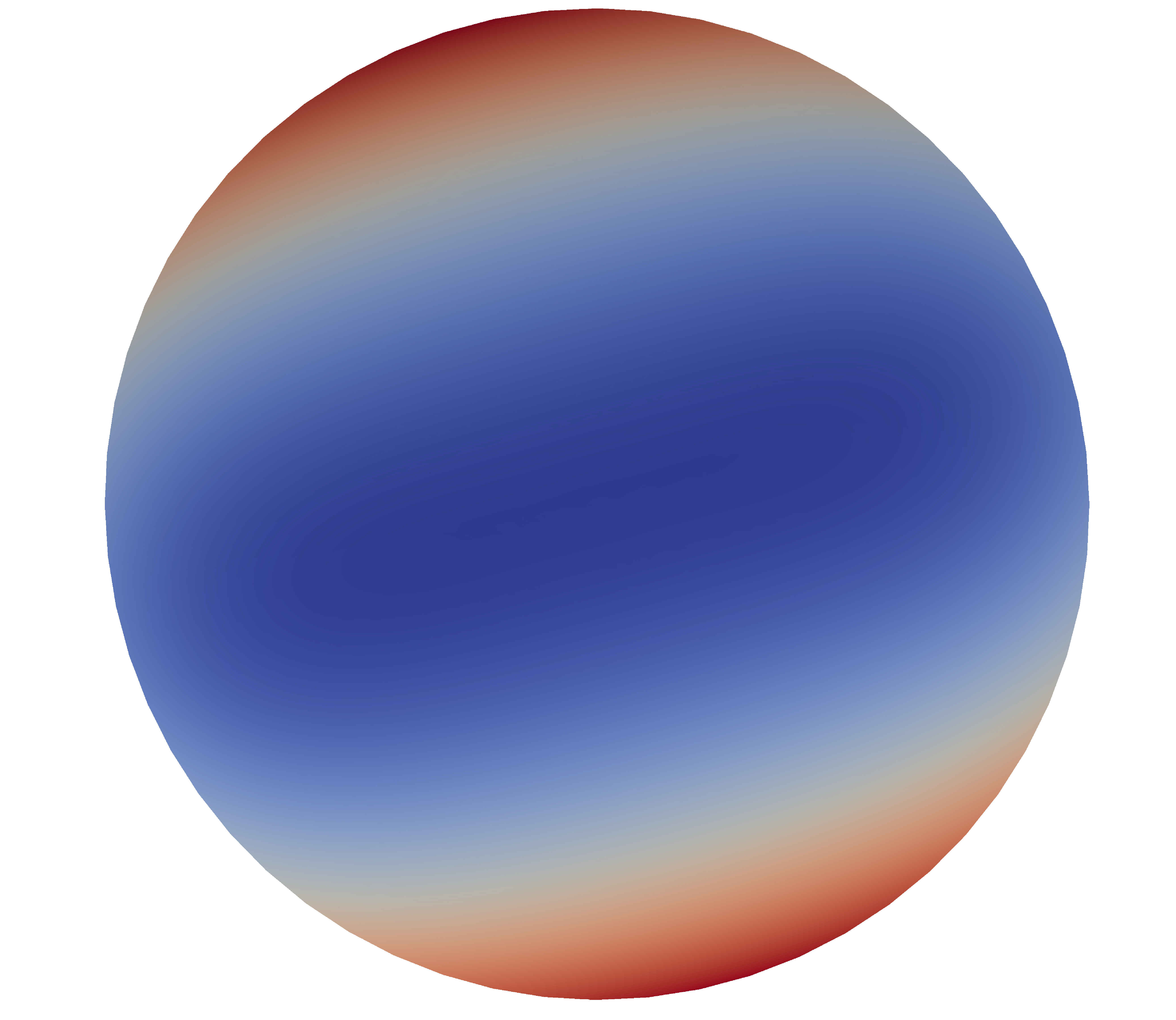}}
  \end{tabular}
  \caption{\label{fig:minimisers-zero}Final configurations after gradient
  descent starting with a flat disk viewed from the top. From left to right,
  top to bottom: $\theta = 1, 81, 91$ and $150$. Color represents the
  magnitude of the displacements $| w |$, from blue at its minimum to red at
  the maximum.}
\end{figure}

We further highlight the behaviour of the solution as a function of $\theta$
in Figure \ref{fig:theta-strains-zero}. In the first plot we compute the mean
bending strains
\[ \frac{1}{| \omega |}  \int_{\omega} (\nabla^2 v)_{i \nocomma i} \mathd x
   \text{\quad with } i \in \{ 1, 2 \} . \]
As mentioned, these act as an easy to compute proxy for the (mean) principal
curvatures. We observe how as $\theta$ increases both strains decrease almost
by an equal amount as the body gradually opens up and flattens out, while
retaining its radial symmetry. However, around $\theta \approx 86$ a stark
change takes place and one of the principal strains decreases while the other
increases. This reflects the abrupt change of the minimiser to a cylindrical
shape. We observe the same phenomenon with the quotient of the principal
axes of the deformed disk in the right plot of the same Figure.

\begin{figure}[h]
  \centering
  \includegraphics[width=6cm]{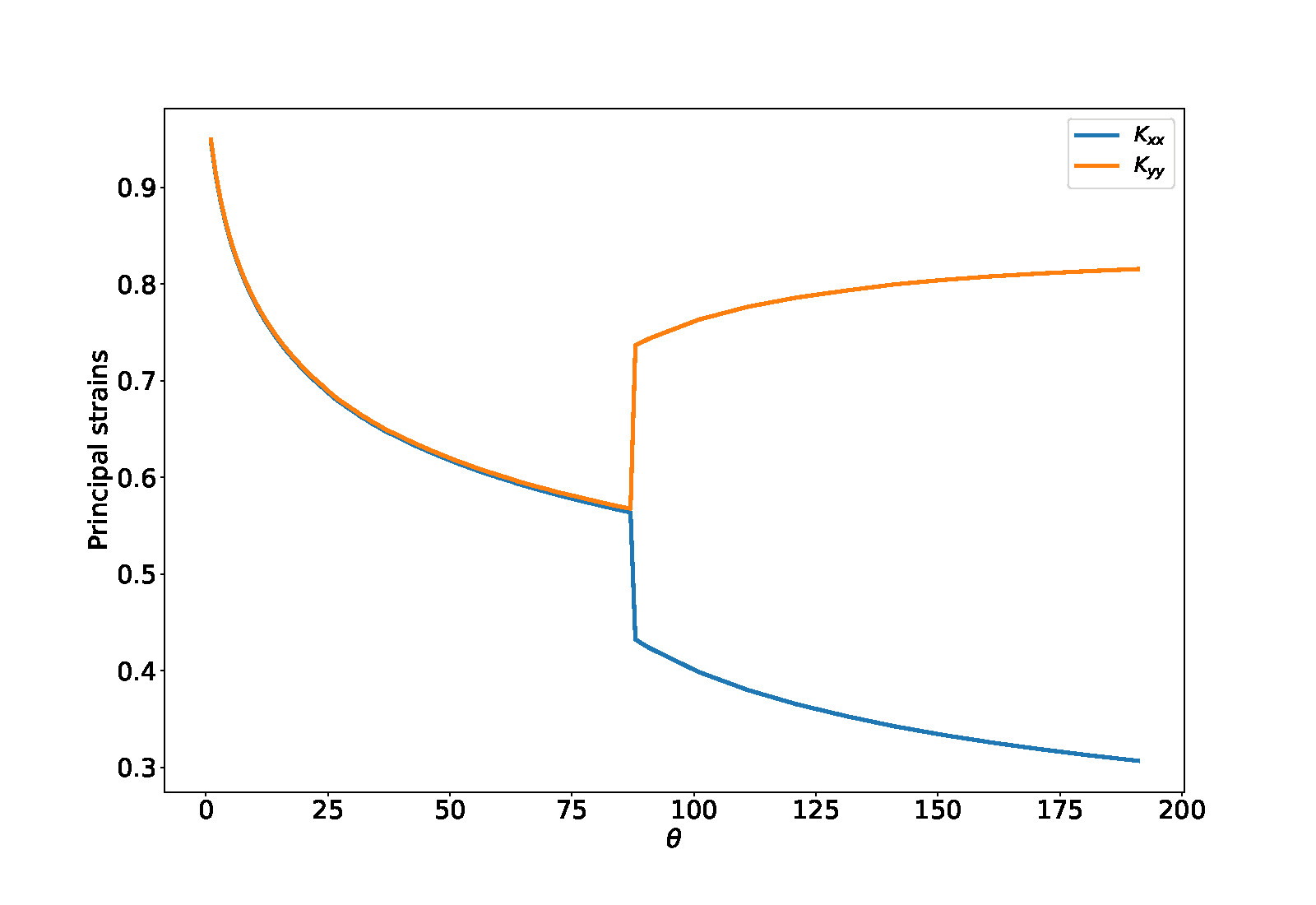}
  \includegraphics[width=6cm]{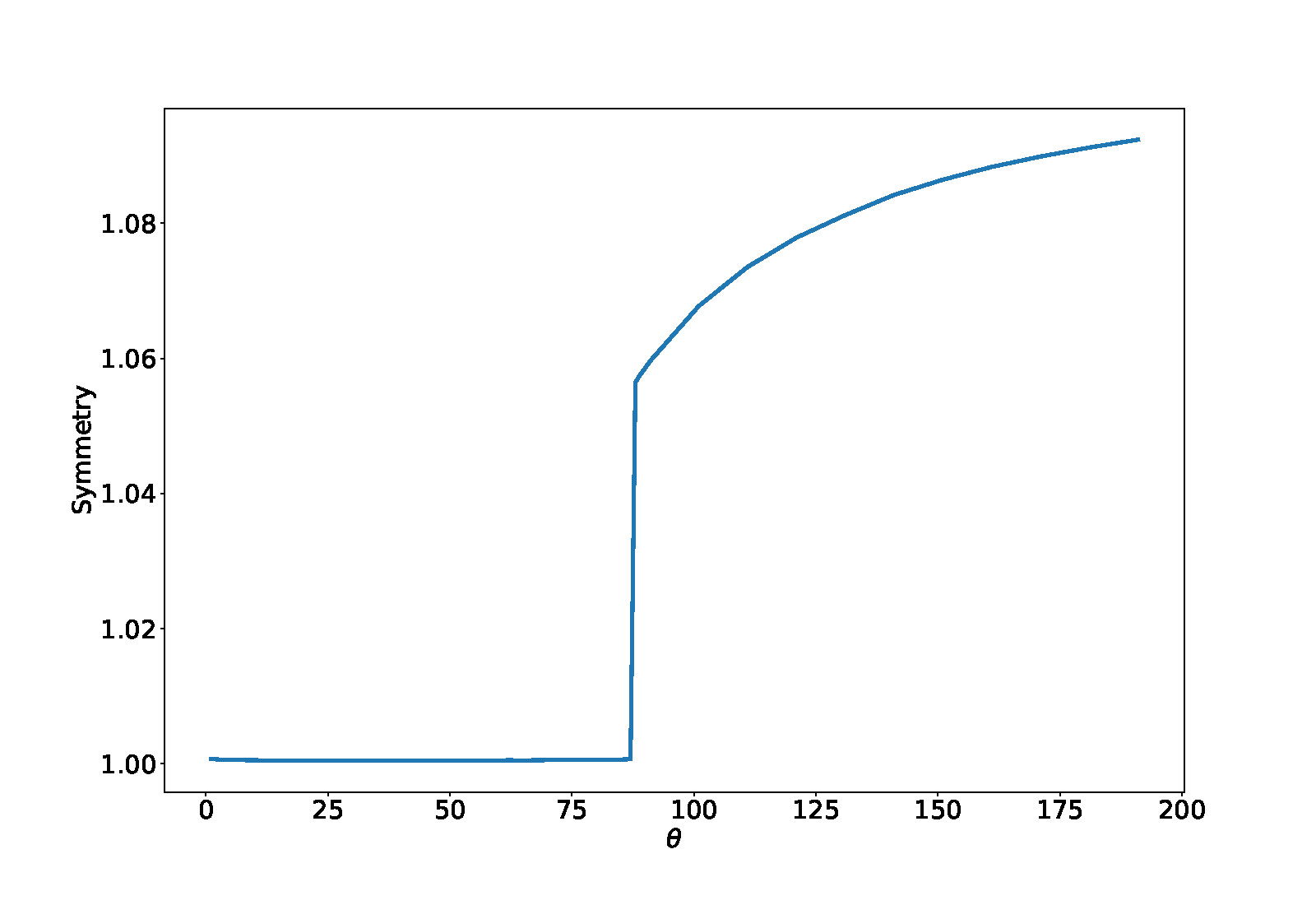}
  \caption{\label{fig:theta-strains-zero}Mean principal strains (left) and symmetry (right) of the
  minimiser as a function of $\theta$ for the flat disk.}
\end{figure}

The second initial condition tested is an orthotropically skewed paraboloid.
Basically, a spherical cap is pressed from the sides to obtain a ``potato
chip''. Testing this shape will highlight the effect of the initial configuration on
the final curvature. We examine its strains and symmetry in Figure

\ref{fig:theta-strains-ani-parab}.

\begin{figure}[h]
  \centering
  \begin{tabular}{ccc}
    \multirow{2}{*}{\resizebox{0.4\columnwidth}{!}{\includegraphics{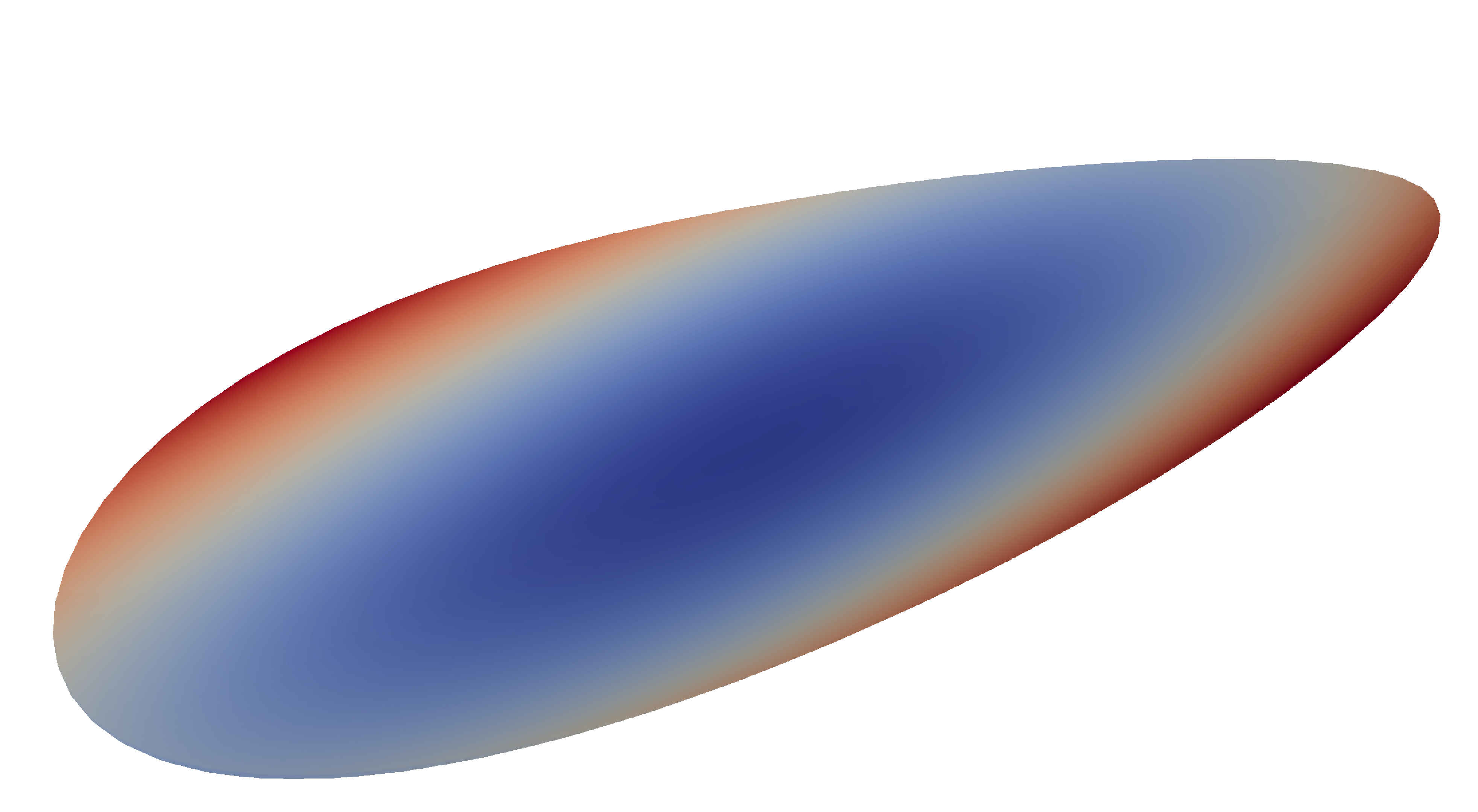}}}
    &
    \resizebox{0.25\columnwidth}{!}{\includegraphics{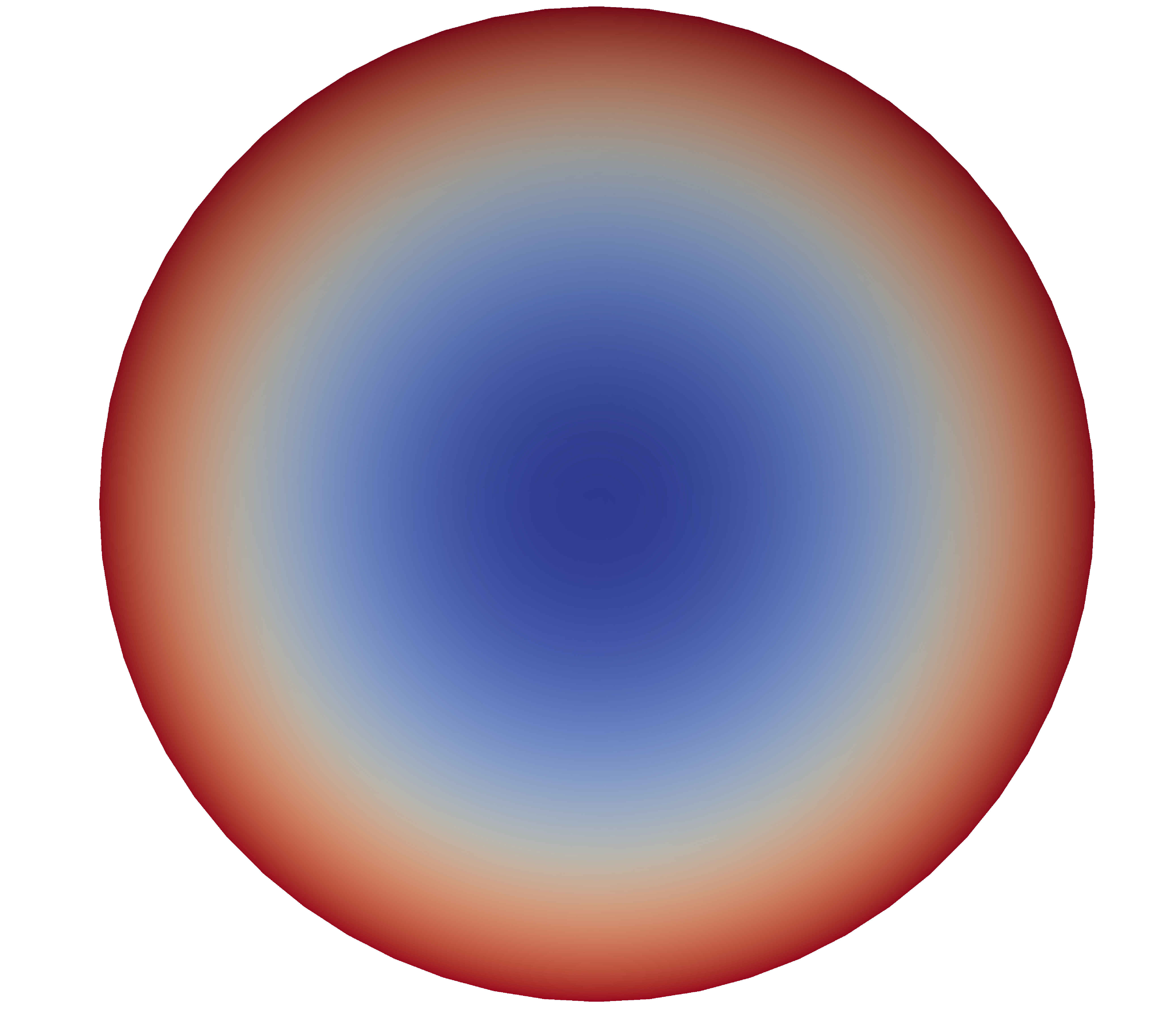}}
    &
    \resizebox{0.25\columnwidth}{!}{\includegraphics{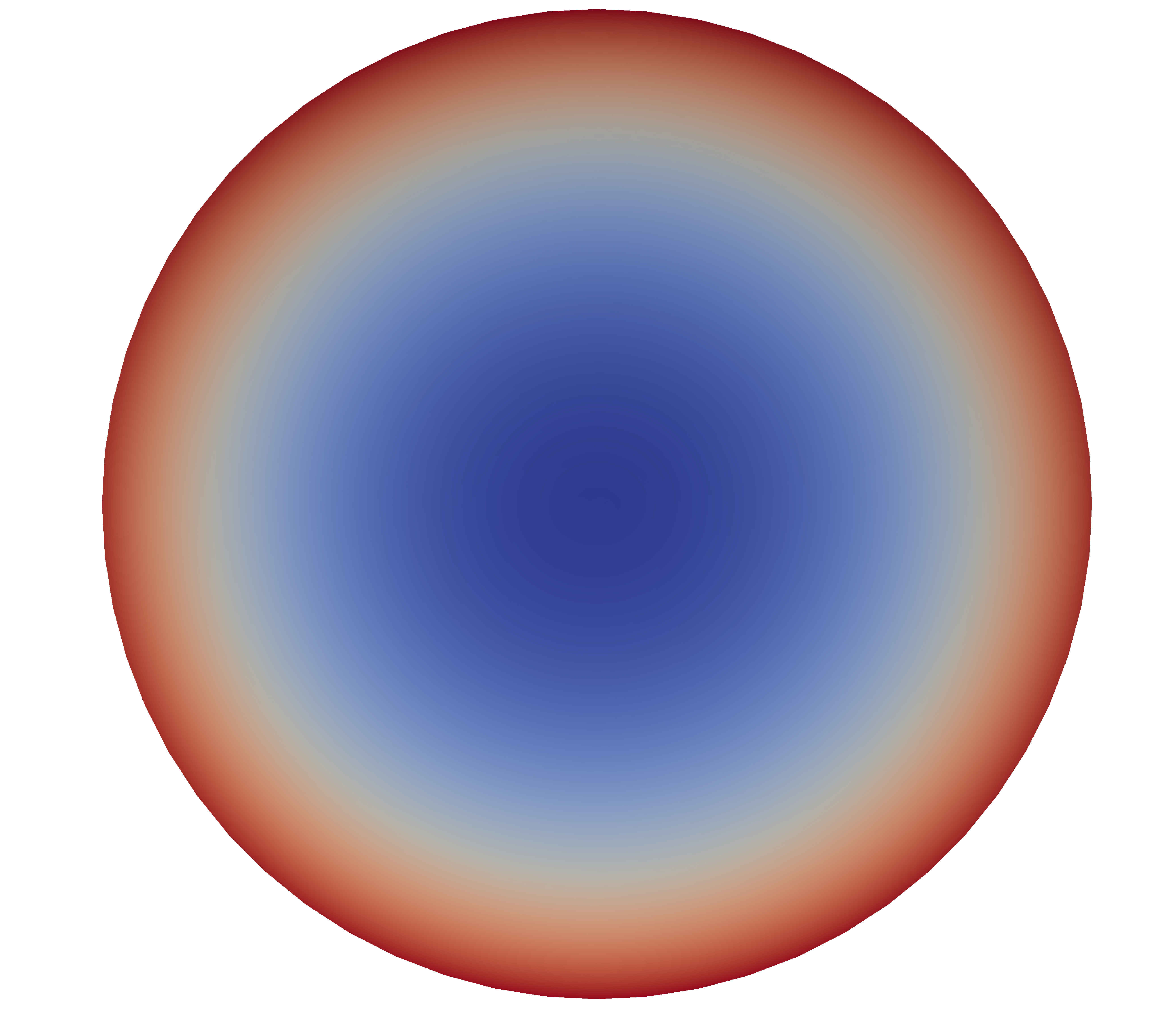}}\\
    &
    \resizebox{0.25\columnwidth}{!}{\includegraphics{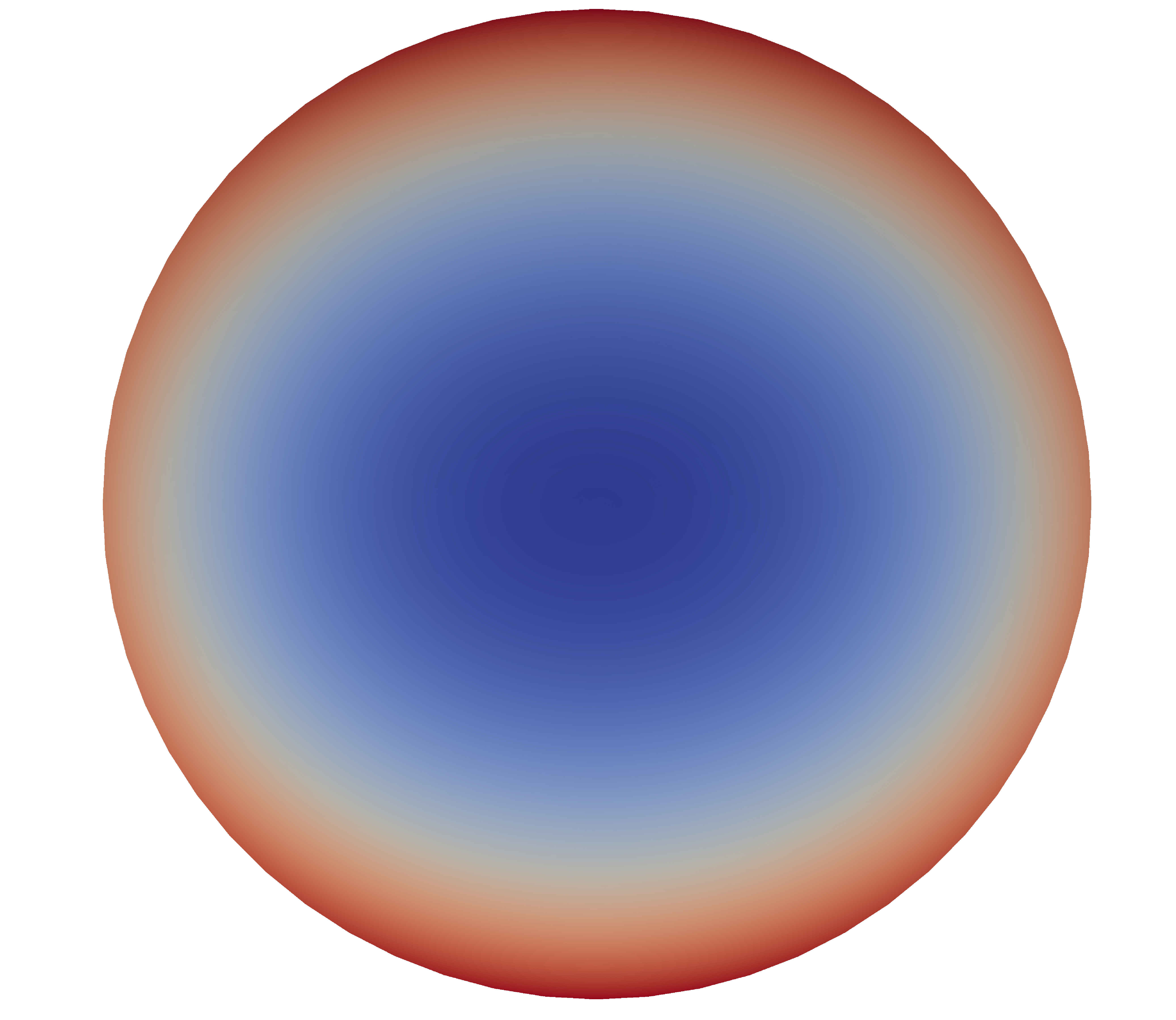}}
    &
    \resizebox{0.25\columnwidth}{!}{\includegraphics{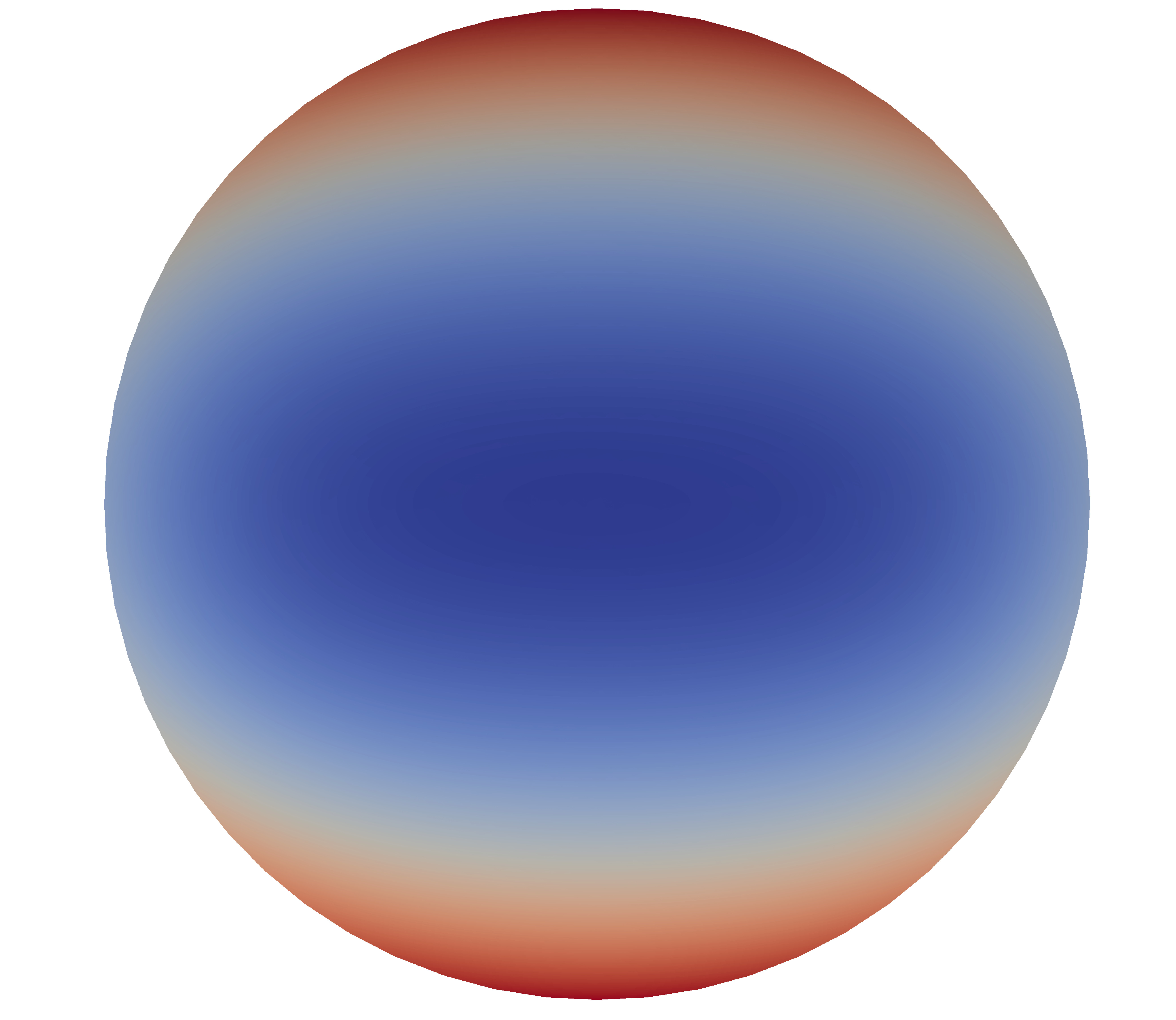}}
  \end{tabular}
  \caption{Initial (left) and final (right) states starting with skewed paraboloid. From left to right,
  top to bottom, $\theta = 1, 51, 61$ and $91$.}
\end{figure}

Again there is a critical value of $\theta \approx 50$ around which the shape
of the minimiser drastically changes. Note however how the change is now
gradual and we see intermediate shapes.

\

\begin{figure}[h]
  \centering
  \includegraphics[width=6cm]{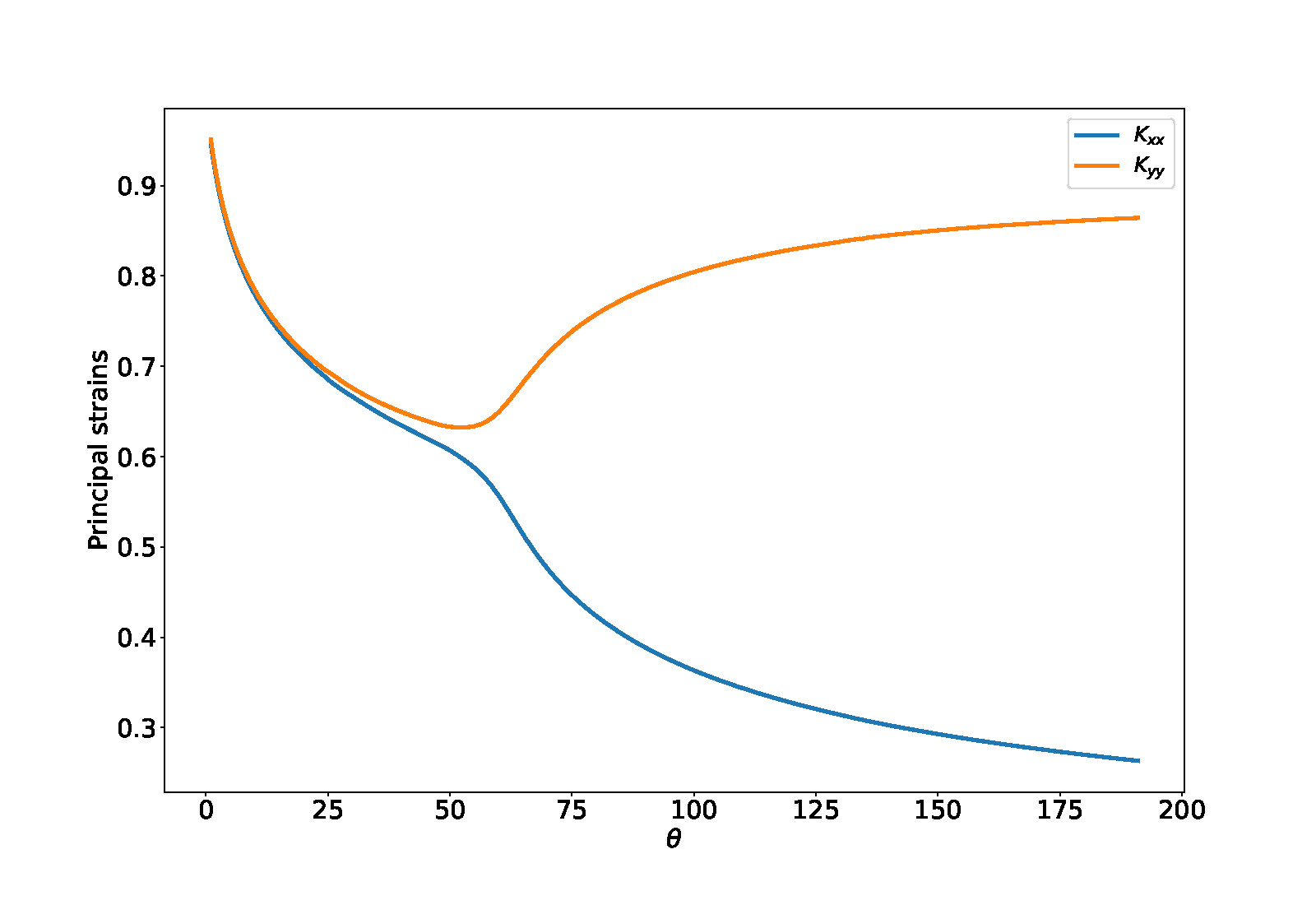}
  \includegraphics[width=6cm]{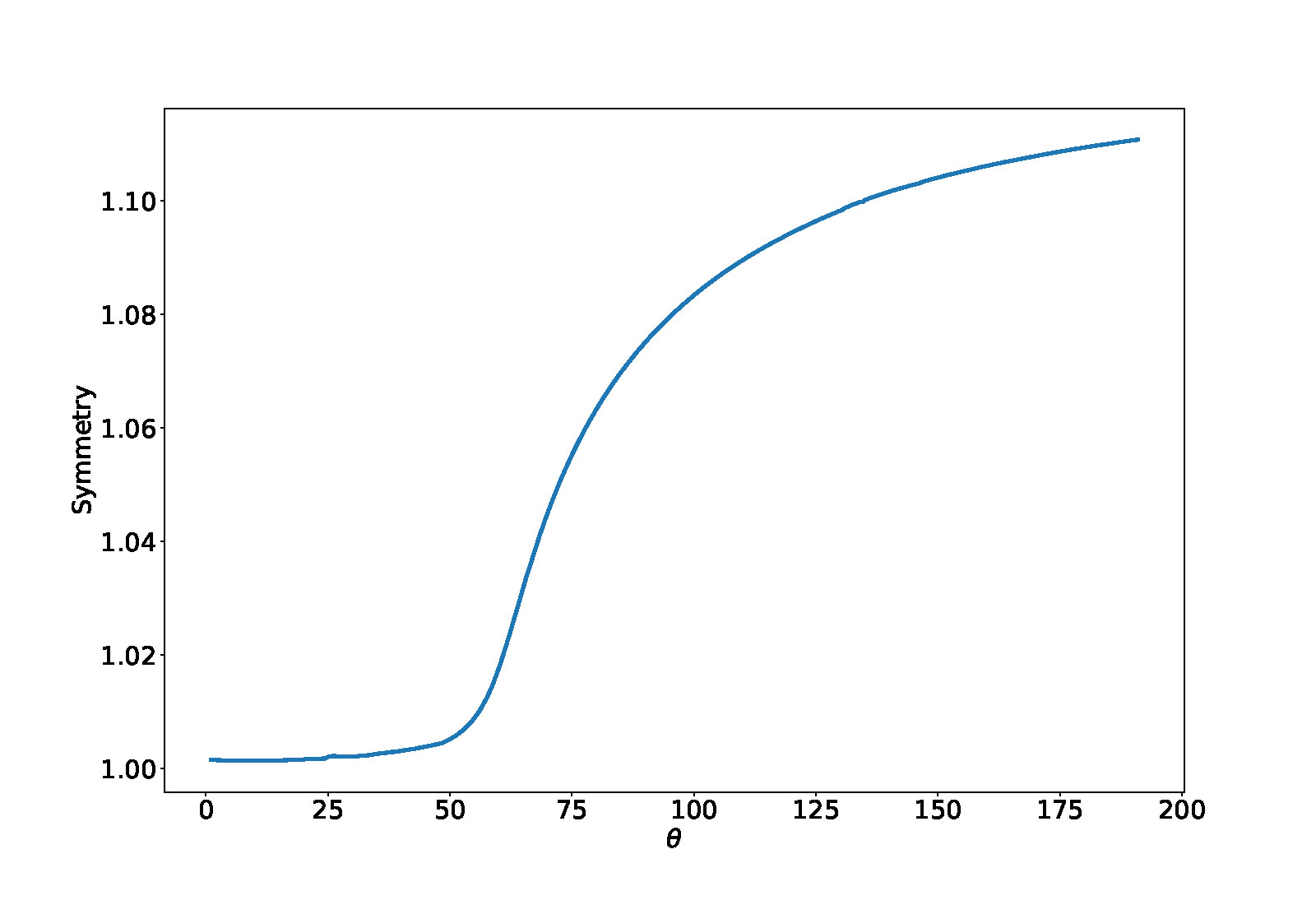}
  \caption{\label{fig:theta-strains-ani-parab}Mean principal strains (left) and symmetry (right) of the
  minimiser as a function of $\theta$ for the skewed paraboloid.}
\end{figure}

%\newpage

\FloatBarrier

%-----------------------------------------------------------------------------------------
\section*{Acknowledgements}

This work was financially supported by project 285722765 of the Deutsche Forschungsgemeinschaft (DFG, German Research Foundation), ``{\tmem{Effektive Theorien und Energie minimierende Konfigurationen für
heterogene Schichten}}''.

%-----------------------------------------------------------------------------------------

\bibliography{hierarchy} 

\bibliographystyle{abbrv}

\end{document}